\documentclass{amsart}

\usepackage{amssymb,amsthm,amsmath,amsfonts,graphicx,color}
\usepackage[colorlinks = true,linktoc = page]{hyperref}

\newtheorem{thm}{Theorem}[section]
\newtheorem{lem}[thm]{Lemma}
\newtheorem{prop}[thm]{Proposition}
\newtheorem{cor}{Corollary}

\newtheorem{defn}{Definition}[section]

\newtheorem{exmp}{Example}[section]

\newtheorem{rem}{Remark}
\newtheorem{note}{Note}

\numberwithin{equation}{section}

\newcommand{\Symb}{\mathcal{A}}
\newcommand{\N}{\mathbb{N}}
\newcommand{\Z}{\mathbb{Z}}
\newcommand{\R}{\mathbb{R}}

\newcommand{\Leng}{\mathcal{L}}
\newcommand{\past}{\mathcal{P}}

\newcommand{\dist}{\mathrm{dist}}
\newcommand{\supp}{\mathrm{supp}}
\newcommand{\diam}{\mathrm{diam}}
\newcommand{\block}{\mathrm{B}}
\newcommand{\romb}{\mathrm{R}}
\newcommand{\camino}{\mathrm{P}}
\newcommand{\neig}{\mathrm{N}}
\newcommand{\Cmu}{\mathrm{c}_\mu}

\providecommand{\argmin}{\mathop{\rm arg\,min}}
\providecommand{\argmax}{\mathop{\rm arg\,max}}

\allowdisplaybreaks

\title[The topological strong spatial mixing property]{The topological strong spatial mixing property\\and new conditions for pressure approximation}
\author{Raimundo Brice\~no}
\address{
Raimundo Brice\~no				\\
Department of Mathematics		\\
The University of British Columbia	\\
1984 Mathematics Road			\\
V6T 1Z2						\\
Vancouver, B.C.				\\
Canada}
\email{raimundo@math.ubc.ca}

\begin{document}

\maketitle

\begin{abstract}
In the context of stationary $\Z^d$ nearest-neighbour Gibbs measures $\mu$ satisfying strong spatial mixing, we present a new combinatorial condition (the topological strong spatial mixing property (TSSM)) on the support of $\mu$ sufficient for having an efficient approximation algorithm for topological pressure. We establish many useful properties of TSSM for studying strong spatial mixing on systems with hard constraints. We also show that TSSM is, in fact, necessary for strong spatial mixing to hold at high rate. Part of this work is an extension of results obtained by D. Gamarnik and D. Katz (2009), and B. Marcus and R. Pavlov (2013), who gave a special representation of topological pressure in terms of conditional probabilities.
\end{abstract}

\tableofcontents

\section{Introduction}

The main goal of this paper is twofold. First, we aim to represent and compute quantitative properties in discrete systems coming from two closely related areas: symbolic dynamics and statistical mechanics. Both share a common ground with different emphasis, which is the study of measures on graphs (typically, a lattice such as $\Z^d$) where the vertices take values on a finite set of \emph{letters} (or \emph{spins}). Secondly, to define and study useful combinatorial conditions for working with such measures on supports with hard constraints, i.e. with local restrictions on the possible configurations.

The quantitative properties considered here are \emph{topological entropy} and its generalization, \emph{topological pressure} (also known as \emph{free energy}, especially in the statistical mechanics context). The two appear in several subjects and, roughly speaking, both try to capture the complexity of a given system by associating to it a nonnegative real number. These values can be represented in several ways: sometimes as a closed formula, other times as a limit and, in the cases of our interest, as the integral of a conditional probability distribution or as the output of an algorithm. Often, it is a difficult task to compute them. In fact, there are computability constraints for approximating these numbers that in general cannot be overcome (e.g. see the characterization of $\Z^d$ topological entropies when $d \geq 2$ in \cite{1-hochman}). We restrict our attention to the subclass of \emph{Markov random fields (MRFs)} known as \emph{nearest-neighbour (n.n.) Gibbs measures}, which are measures defined through local spin interactions. 

In the context of n.n. Gibbs measures, there has been a growing interest \cite{1-martinelli,1-weitz,1-goldberg,2-gamarnik} in a property exhibited by some of these measures known as \emph{strong spatial mixing (SSM)}. This property, related to the absence of a ``boundary phase transition'' \cite{1-martinelli}, is physically meaningful and has proven to be useful in the development of approximation algorithms (e.g. counting independent sets \cite{1-weitz}). It is also a stronger version of a property called \emph{weak spatial mixing (WSM)}, related with uniqueness of equilibrium states \cite{2-weitz} and the absence of a ``phase transition''. Examples of systems that satisfy these properties in some regime include the Ising and Potts models \cite{1-martinelli,2-goldberg}, and even some cases where hard constraints are considered, such as the hard-core model \cite{1-weitz} and $k$-colourings \cite{2-gamarnik} (here called \emph{$k$-checkerboards}). In this paper we study some characteristics that a set of hard constraints should satisfy in order to be (to some extent) compatible with SSM. We introduce a new property on the support of MRFs here called \emph{topological strong spatial mixing (TSSM)}, because of its close relationship with its measure-theoretic counterpart and the absence of a ``combinatorial boundary phase transition''. 

Following the works of D. Gamarnik and D. Katz \cite{1-gamarnik}, and B. Marcus and R. Pavlov \cite{1-marcus}, we provide extended versions of representation theorems of topological pressure in terms of conditional probabilities and also conditions for more general approximation algorithms. In \cite{1-gamarnik}, for obtaining such representation and approximation theorems, they assumed a very strong combinatorial condition that here we call \emph{safe symbol}, together with SSM (and an exponential assumption on the rate of SSM for algorithmic purposes). Later, in \cite{1-marcus}, this assumption was replaced by more general and technical conditions in the case of representation, and a property called there \emph{single-site fillability (SSF)}, which generalized the safe symbol case both in the representation and in the algorithmic results. Here, making use of the theoretical machinery developed in \cite{1-marcus}, we have relaxed even more those conditions, by using the more general property of TSSM and extending the representation and algorithmic results to a point that sometimes can be regarded as optimal. By optimal, we mean (and prove) that TSSM is in some instances a necessary condition for SSM to hold.

Other combinatorial, topological and measure-theoretic mixing properties have been considered in the literature of lattice systems. We also explore the relationships between some of them and how they have shown to be useful in some cases for representation and approximation.

Summarizing, we focus on:
\begin{enumerate}
\item properties of qualitative mixing conditions and relationships among them,
\item representation of topological entropy/pressure through useful formulas, and
\item algorithms for approximating such quantities.
\end{enumerate}

The paper is organized as follows: First, in Section \ref{sec2} and Section \ref{sec4}, we introduce the basic notions of symbolic dynamics, MRFs, Gibbs measures and mixing properties. Then, in Section \ref{tssmproperties}, we define the notion of TSSM, establish characterizations of it and relationships to measure-theoretic quantities. Next, in Section \ref{sec6}, we provide connections between measure-theoretic and combinatorial mixing properties; in particular, Theorem \ref{rate} provides evidence that TSSM is closely related with SSM. In Section \ref{examples}, we give several examples illustrating different kinds of mixing properties. Among these examples, we consider the $\Z^2$ $4$-checkerboard and prove that is not possible to have a Gibbs measure supported on it satisfying SSM (it has been suggested in the literature \cite{1-salas} that a uniform Gibbs measure on this system should satisfy WSM). Finally, in Section \ref{PressureRep} and Section \ref{PressureApp}, we discuss some pressure representation theorems and we show how a good representation can be used for developing efficient approximation algorithms in a similar fashion to \cite{1-marcus}.

Many of the results in this work we believe are easily extendable to other regular infinite graphs (transitive, Cayley, etc.) besides $\Z^d$.
 
\section{Definitions and preliminaries}
\label{sec2}

Given $d \in \N$, consider the \emph{$d$-dimensional cubic lattice} $\Z^d$, a finite set of \emph{letters} $\Symb$ called the \emph{alphabet}, and the space of arrays $\Symb^{\Z^d}$, the \emph{full shift}. Notice that $\Z^d$ (in a slight abuse of notation) can be regarded as a countable graph with regular degree $2d$, where $\mathcal{V}(\Z^d) = \Z^d$ is the set of vertices (or \emph{sites}) and $\mathcal{E}(\Z^d) = \left\{e = \{p,q\}: p,q \in \Z^d, \|p-q\|=1\right\}$ is the set of edges, with $\|p\| = \sum_{i=1}^{d}\left|p_i\right|$. Given $p \in \Z^d$, denote $\sigma:\Z^d \times \Symb^{\Z^d} \to \Symb^{\Z^d}$ the natural \emph{shift action} on $\Symb^{\Z^d}$ defined by $(p,x) \mapsto \sigma_p(x)$, with $\left(\sigma_p(x)\right)(q) = x(p+q)$, for $q \in \Z^d$. Considering the distance function $m(x,y) = 2^{-\inf\{\|p\|: x(p) \neq y(p)\}}$, $(\Symb^{\Z^d},m)$ is a compact metric space.

We will denote all subsets of $\Z^d$ with uppercase letters (e.g. $S$, $T$, etc.). Whenever a finite set $S$ is contained in an infinite set $T$, we denote this by $S \Subset T$. The \emph{(outer) boundary} of $S \subseteq \Z^d$ is the set $\partial S$ of $p \in \Z^d \backslash S$ which are adjacent to some element of $S$, i.e. $\partial S := \left\{p \in S^c: \dist(\{p\},S) = 1\right\}$, where $\dist(A,B) = \min_{p \in A, q \in B} \|p - q\|$, for $A,B \subseteq \Z^d$. When denoting subsets of $\Z^d$ that are singletons, brackets will be usually omitted, e.g. $\dist(\{p\},S)$ will be regarded to be the same as $\dist(p,S)$. We will say that two sites $p,q \in \Z^d$ are adjacent if $\dist(p,q) = 1$ and we will denote this by $p \sim q$.

Given $n \in \N$, $\neig_n(S) := \left\{p \in \Z^d: \dist(p,S) \leq n\right\}$ denotes the \emph{$n$-neighbourhood} of $S$ and $\partial_n(S) := \neig_n(S) \backslash S$, the \emph{$n$-boundary} of $S$. The \emph{$n$-block} is the set $\block_n := \{p \in \Z^d: |p_i| \leq n, \mbox{ for all } i\}$ and the \emph{$n$-rhomboid}, $\romb_n := \{p \in \Z^d: \|p\| \leq n\} = \neig_n(0)$, where $0$ denotes the zero vector with number of coordinates depending on the context. For a finite set $S \Subset \Z^d$, we define its \emph{diameter} as $\diam(S) := \max_{p,q \in S}\dist(p,q)$. A \emph{path} will be any sequence $\camino \Subset \Z^d$ of distinct vertices $p_1, p_2, \dots, p_{n}$ such that $\dist(p_i,p_{i+1}) = 1$, for all $1 \leq i < n$, with $|\camino| = n$. For $W \subseteq \Z^d$ not containing a site $p \in \Z^d$, a \emph{path from $p$ to $W$} is a path whose first vertex is $p$ and whose last vertex is in $\partial W$. A set $S \subseteq \Z^d$ is said to be \emph{connected} if for every $p,q \in S$, there is a path $\camino$ from $p$ to $q$ contained in $S$ (i.e. $\camino \subseteq S$).

A \emph{configuration} is a map $u: T \to \Symb$ for $\emptyset \neq T \subseteq \Z^d$ (i.e. $u \in \Symb^T$), which will be usually denoted with lowercase letters (e.g. $u$, $v$, etc.). $T$ is called the \emph{shape} of $u$, and a configuration will be said to be finite if its shape is finite. For any configuration $u$ with shape $T$ and $S \subseteq T$, $u(S)$ denotes the restriction of $u$ to $S$, i.e. the \emph{sub-configuration} of $u$ occupying $S$. For $S$ and $T$ disjoint sets, $u \in \Symb^S$ and $v \in \Symb^T$, $uv$ will be the configuration on $S \cup T$ defined by $(uv)(S) = u$ and $(uv)(T) = v$, called the \emph{concatenation} of $u$ and $v$. A \emph{point} is a configuration with shape $\Z^d$, usually denoted with letters $x$, $y$, etc.

Given a countable family $\mathcal{F}$ of finite configurations, define:
\begin{equation}
\mathsf{X}(\mathcal{F}) := \left\{x \in \Symb^{\Z^d}: \sigma_p(x)(S) \notin \mathcal{F}, \mbox{ for all } S \Subset \Z^d, \mbox{ for all } p \in \Z^d\right\}.
\end{equation}

Here, $X = \mathsf{X}(\mathcal{F}) \subseteq \Symb^{\Z^d}$ is called a \emph{$\Z^d$ shift space} and is the set of all points that do not contain an element from $\mathcal{F}$ as a sub-configuration, up to translation. Notice that a shift space $X$ is always a shift-invariant set, i.e. $\sigma_p(X) = X$, for all $p \in \Z^d$. In fact, a subset $X \subseteq \Symb^{\Z^d}$ is a shift space if and only if it is shift-invariant and closed for the metric $m$. More than one family $\mathcal{F}$ can define the same shift space $X$ and in the case where $X$ can be defined by a finite family $\mathcal{F}$, it is said to be a \emph{shift of finite type (SFT)}. An SFT is a \emph{nearest-neighbour (n.n.) SFT} if $\mathcal{F}$ can be chosen to be configurations only on shapes on edges, i.e. pairs of the form $\{p,p+e_i\}$, where $p \in \Z^d$ and $\{e_i\}_{i=1}^{d}$ denote the canonical basis. Along this paper, we restrict our attention to n.n. SFTs in almost every case.

\begin{exmp}
Two important examples of n.n. SFT are the \emph{hard square $\Z^d$ shift space $\mathcal{H}_d$} (the shift space of points in $\{0,1\}^{\Z^d}$ with no adjacent $1$s) and the \emph{$k$-checkerboards} $\mathcal{C}_d(k)$ (the shift spaces of proper $k$-colourings on $\Z^d$), where $k \geq 2$. Formally,
\begin{align}
\mathcal{H}_d	    &	:= \left\{x \in \{0,1\}^{\Z^d}: x(p) \cdot x(p+e_i) = 0, \mbox{ for all } p \in \Z^d, i = 1,\dots,d\right\},		\\
\mathcal{C}_d(k)  &	:= \left\{x \in \{1,\dots,k\}^{\Z^d}: x(p) \neq x(p+e_i) , \mbox{ for all } p \in \Z^d, i = 1,\dots,d\right\}.
\end{align}
\end{exmp}

The \emph{language} of a shift space $X$ is:
\begin{equation}
\Leng(X) := \bigcup_{S \Subset \Z^d} \Leng_S(X),
\end{equation}
where $\Leng_S(X) := \left\{x(S): x \in X\right\}$. For a subset $S \subseteq \Z^d$, a configuration $u \in \Symb^S$ is \emph{globally admissible} for $X$ if $u$ extends to a point on $\Z^d$, i.e. if there exists $x \in X$ such that $x(S) = u$. So, the language $\Leng(X)$ is precisely the set of finite globally admissible configurations. On the other hand, given $S \subseteq \Z^d$ and a configuration $u \in \Symb^S$, we denote $[u]_X := \left\{x \in X: x(S) = u\right\}$. When $S$ is finite, these sets are called \emph{cylinder sets}, and when omitting the subscript $X$, we will think of $[u]$ as the cylinder for the full shift $X = \Symb^{\Z^d}$.

\subsection{Conjugacy and topological entropy}
A natural way to transform one shift space to another is via a particular class of maps given by the following definition. 

\begin{defn}
\label{conj}
A \emph{sliding block code} between $\Z^d$ shift spaces $X$ and $Y$ is a map $\phi:X \to Y$ for which there is a positive integer $N$ and a map $\Phi:\Leng_{\block_N}(X) \to \Leng_{\block_1}(Y)$ such that:
\begin{equation}
\phi(x)_p = \Phi(x(p+\block_N)), \mbox{ for all } p \in \Z^d.
\end{equation}

A \emph{conjugacy} is an invertible sliding block code, and two shift spaces $X$ and $Y$ are said to be \emph{conjugate} (denoted $X \cong Y$) if there is a conjugacy from one to the other.
\end{defn}

\begin{exmp}
\label{blockcode}
Given $N \in \N$ and a shift space $X \subseteq \Symb^{\Z^d}$, a natural sliding block code is the \emph{higher block code} $\beta_N: X \rightarrow \left(\Symb^{\block_N}\right)^{\Z^d}$ defined by:
\begin{equation}
\left(\beta_N(x)\right)_p = x(p+\block_N).
\end{equation}

We call the image, $Y = \beta_N(X)$, a \emph{higher block code representation} of $X$. Notice that the alphabet of $Y$ is $\Symb^{\block_N}$.
\end{exmp}

Two shift spaces are often regarded as being the same if they are conjugate. Properties preserved by conjugacies are called \emph{conjugacy invariants}. For example, the property of being an SFT is a conjugacy invariant: If a shift space $X$ is conjugate to an SFT, then $X$ itself is an SFT. Another important invariant is the following.

\begin{defn}
The \emph{topological entropy} of a shift space $X$ is defined as:
\begin{equation}
h(X) := \inf_{n} \frac{\log\left|\Leng_{\block_n}(X)\right|}{|\block_n|} = \lim_{n \rightarrow \infty} \frac{\log\left|\Leng_{\block_n}(X)\right|}{|\block_n|}.
\end{equation}
\end{defn}

Topological entropy is a conjugacy invariant (i.e. if $X \cong Y$, then $h(X) = h(Y)$). The limit always exists because $\{\left|\Leng_{\block_n}(X)\right|\}_n$ is a (coordinate-wise) sub-additive sequence and a well-known multidimensional extension of Fekete's sub-additive lemma applies \cite{1-balister}. Notice that the topological entropy can be regarded as the growth rate of globally admissible configurations on $\block_n$.

It is important to point that for every SFT there is a n.n. SFT higher block code representation. In the case of n.n. SFTs, there is a simple algorithm for computing $h(X)$ when $d = 1$, because $h(X) = \log\lambda_A$, for $\lambda_A$ the largest eigenvalue of the adjacency matrix $A$ of the edge shift representation of $X$ \cite{1-marcus}. However, for $d \geq 2$, there is in general no known closed form for the entropy. Only in a few specific cases a closed form is known (e.g. dimer model, square ice \cite{1-kasteleyn,1-lieb}).

\begin{exmp}
For $\mathcal{H}_1 \subseteq \{0,1\}^{\Z}$, it is easy to see that $h(\mathcal{H}_1) = \log \lambda$, where $\lambda \approx 1.68103$ is the golden ratio. On the other hand, no closed form is known for the value of $h(\mathcal{H}_d)$, for $d \geq 2$.
\end{exmp}

One can hope to approximate the value of the topological entropy of a multidimensional SFT, whether by using its definition and truncating the limit or by alternative methods. A relevant fact is that, for $d \geq 2$, it is algorithmically undecidable to know if a given configuration is in $\Leng(X)$ or not \cite{1-berger,1-robinson}. In this sense, it is useful to define an alternative, still meaningful, set of configurations. Given a family of configurations $\mathcal{F}$ and a shape $S$, $u \in \Symb^S$ is said to be \emph{locally admissible} for $X = \mathsf{X}(\mathcal{F})$ if for all $S' \subseteq S$, $u(S') \notin \mathcal{F}$, up to translation. Notice that a point $x$ is locally admissible if and only if $x$ is globally admissible. The set of finite locally admissible configurations will be denoted by $\Leng^{\rm l.a.}_S(\mathcal{F})$. Considering this, we have the following result.

\begin{thm}[\cite{1-friedland,1-hochman}]
\label{thm-friedland}
Given a finite family of configurations $\mathcal{F}$ and the SFT $X = \mathsf{X}(\mathcal{F})$, $h(X)$ can be computed by counting locally admissible configuration rather than globally admissible ones:
\begin{equation}
h(X) =  \inf_{n}\frac{\log \left| \Leng^{\rm l.a.}_{\block_n}(\mathcal{F}) \right|}{|\block_n|} = \lim_{n \rightarrow \infty}\frac{\log \left| \Leng^{\rm l.a.}_{\block_n}(\mathcal{F}) \right|}{|\block_n|}.
\end{equation}
\end{thm}

Since counting locally admissible configurations is tractable, it can be said that Theorem \ref{thm-friedland} already provides an approximation algorithm for the topological entropy of an SFT. Formally, a real number $h$ is \emph{right recursively enumerable} if there is a Turing machine which, given an input $n \in \N$, computes a rational number $r(n) \geq h$ such that $r(n) \rightarrow h$. Given Theorem \ref{thm-friedland} and the fact that such limit is also an infimum, we can see that $h(X)$ is right recursively enumerable, for any $\Z^d$ SFT $X$. In fact, the converse is also true due to the following celebrated result from M. Hochman and T. Meyerovitch.

\begin{thm}[\cite{1-hochman}]
The class of right recursively enumerable numbers is exactly the class of entropies of $\Z^d$ SFTs.
\end{thm}

A real number $h$ is \emph{computable} if there is a Turing machine which, given an input $n \in \N$, computes a rational number $r(n)$ such that $|h - r(n)| < \frac{1}{n}$. For example, every algebraic number is computable, since there are numerical methods for approximating the roots of an integer polynomial. This is a strictly stronger notion than right recursively enumerable \cite{1-ko}. It can be shown that, under extra (mixing) assumptions on an SFT $X$, $h(X)$ turns out to be computable (see \cite{1-hochman} and Theorem \ref{mixComp}). Moreover, the difference $|h - r(n)|$ can be thought as a function of $n$, introducing a refinement of the classification of entropies by considering the speed of approximation. A relevant case for us is when that function is bounded by a polynomial in $\frac{1}{n}$.

\begin{exmp}[\cite{1-pavlov,1-gamarnik}]
\label{hard2}
The topological entropy $h(\mathcal{H}_2)$ of the hard square $\Z^2$ shift space is a computable number that can be approximated in polynomial time.
\end{exmp}

\subsection{Measure-theoretic definitions}

In Example \ref{hard2}, which is basically a combinatorial/topological result, the proofs from \cite{1-pavlov} and \cite{1-gamarnik} are almost entirely based on probabilistic and measure-theoretic techniques. In this paper are frequently considered Borel probability measures $\mu$ on $\Symb^{\Z^d}$. This means that $\mu$ is determined by its values on the cylinder sets $[u]$, where $u$ is a configuration with arbitrary shape $S \Subset \Z^d$. For notational convenience, when measuring cylinder sets, we just use the configuration $u$ instead of $[u]$. For instance, $\mu\left(uv \middle\vert w\right)$ represents the conditional measure $\mu\left([u] \cap [v] \middle\vert [w]\right)$.

A measure $\mu$ on $\Symb^{\Z^d}$ is \emph{shift-invariant} (or \emph{stationary}) if $\mu(\sigma_{p}(C)) = \mu(C)$, for all measurable sets $C$ and $p \in \Z^d$. Given a shift space $X$, $\mathcal{M}(X)$ denotes the set of shift-invariant Borel probability measures whose \emph{support} $\supp(\mu)$ is contained in $X$, where:
\begin{equation}
\supp(\mu) := \left\{x \in \Symb^{\Z^d}: \mu(x(S)) > 0, \mbox{ for all } S \Subset \Z^d\right\}.
\end{equation}

In this context, the support $\supp(\mu)$ turns out to be always a shift-space (closed and shift-invariant). A measure $\mu \in \mathcal{M}(\Symb^{\Z^d})$ is \emph{ergodic} if whenever $C \subseteq \Symb^{\Z^d}$ is measurable and shift-invariant (i.e, if $\sigma_p(C) = C$, for all $p \in \Z^d$), then $\mu(C) \in \{0,1\}$. For $\mu \in \mathcal{M}(\Symb^{\Z^d})$, we can also define a notion of entropy.

\begin{defn}
The \emph{measure-theoretic entropy} of a shift-invariant measure $\mu$ is defined as:
\begin{equation}
h(\mu) := \lim_{n \rightarrow \infty} \frac{-1}{|\block_n|} \sum_{w \in \Symb^{\block_n}} \mu(w)\log(\mu(w)),
\end{equation}
where $0 \log 0 = 0$.
\end{defn}

A fundamental relationship between topological and measure-theoretic entropy is the following.

\begin{thm}[Variational Principle \cite{1-misiurewicz}]
\label{varPrinc}
Given a shift space $X$,
\begin{equation}
h(X) = \sup_{\mu \in \mathcal{M}(X)} h(\mu) = \max_{\mu \in \mathcal{M}(X)} h(\mu).
\end{equation}
\end{thm}

\begin{rem}
The measures that achieve the maximum are called \emph{measures of maximal entropy (m.m.e.)} for $X$. Notice that if $\mu$ is an m.m.e. for $X$, then $h(X) = h(\mu)$.
\end{rem}

Given a shift space $X$ and a continuous function $f \in C(X)$, we define the \emph{topological pressure}, that can be regarded as a generalization of topological entropy.

\begin{defn}
Given a n.n. $\Z^d$ SFT $X$ and $f \in C(X)$, the \emph{topological pressure} of $f$ on $X$ is:
\begin{equation}
P_X(f) := \sup_{\mu \in \mathcal{M}(X)}\left(h(\mu) + \int{f}d\mu\right).
\end{equation}
\end{defn}

In this case, the supremum is also always achieved and any measure which achieves the supremum is called an \emph{equilibrium state} for $X$ and $f$. We write $P(f)$ instead of $P_X(f)$, if $X$ is understood. Notice that in the special case when $f \equiv 0$, $P(f)$ is the topological entropy $h(X)$ of $X$, thanks to Theorem \ref{varPrinc}.

\begin{note}
The preceding definition is a characterization of pressure in terms of a variational principle, but can also be regarded as its definition (see \cite[Theorem 6.12]{1-ruelle}). Informally, topological pressure can be thought as a growth rate, where the configurations are ``weighted'' by the given function. This idea is formalized for a more particular case in the next subsection.
\end{note}

\subsection{Markov random fields and Gibbs measures}

A key family of measures on $\Z^d$ for our purposes is the following one.

\begin{defn}
A shift-invariant\footnote{In general, a measure does not need to be shift-invariant for being an MRF, but in this paper we will always assume shift-invariance.} measure $\mu$ on $\Symb^{\Z^d}$ is a \emph{Markov random field (MRF)} if, for any set $S \Subset \Z^d$, any $u \in \Symb^S$, any $T \Subset \Z^d$ s.t. $\partial S \subseteq T \subseteq \Z^d \backslash S$, and any $\delta \in \Symb^T$ with $\mu(\delta) > 0$, it is the case that:
\begin{equation}
\mu\left(u \middle\vert \delta\right) = \mu\left(u  \middle\vert \delta(\partial S)\right).
\end{equation}
\end{defn}

In other words, an MRF is a measure where every finite configuration conditioned to its boundary is independent of the complement. 

\begin{defn}
Given an MRF $\mu$, a set $S \Subset \Z^d$, and $\delta \in \Symb^{\partial S}$ with $\mu(\delta) > 0$,  $\mu^\delta$ will denote the measure on $\Symb^S$ such that:
\begin{equation}
\mu^{\delta}(u) := \mu\left(u \middle\vert \delta\right),
\end{equation}
for every $S' \subseteq S$ and $u \in \Symb^{S'}$.
\end{defn}

Now we discuss what a Gibbs measure is, though not in its most general form. The main characteristic of the families of measures presented here is their local nature, something that will be useful for developing efficient algorithms. We will deal mostly with (stationary) nearest-neighbour Gibbs measures, which are MRFs specified by nearest-neighbour interactions.

\begin{defn}
A \emph{nearest-neighbour (n.n.) interaction} is a shift-invariant function $\Phi$ from the set of configurations on vertices and edges in $\Z^d$ to $\R \cup \{\infty\}$. Here, shift-invariance means that $\Phi(\sigma_{p}(w)) = \Phi(w)$ for all finite configurations $w$ on edges and vertices, and all $p \in \Z^d$.
\end{defn}

Clearly, a n.n. interaction is defined by only finitely many numbers, namely the values of the interaction on configurations on $\{0\}$ and edges $\{0,e_i\}$, $i = 1,\dots,d$. W.l.o.g., we can assume that the values on vertices are not $\infty$ (if not, we remove such element from $\Symb$). However, it is meaningful to assume that $\Phi$ is $\infty$ on edges because these are what we call \emph{hard constraints}. For a n.n. interaction $\Phi$, we define its \emph{underlying SFT} as:
\begin{equation}
\mathsf{X}(\Phi) := \left\{x \in \Symb^{\Z^d}: \Phi(x(\{p,p+e_i\})) \neq \infty, \mbox{ for all } p \in \Z^d, i=1,\dots,d\right\}.
\end{equation}

Notice that $\mathsf{X}(\Phi)$ is a n.n. SFT.

\begin{defn}
For a n.n. interaction $\Phi$ and a set $S \Subset \Z^d$, the \emph{energy function} $U_{S}^{\Phi}: \Symb^S \to \R \cup \{\infty\}$ is:
\begin{equation}
U_{S}^{\Phi}(w) := \sum_{p \in S} \Phi(w(p)) + \sum_{e \subseteq S} \Phi(w(e)),
\end{equation}
where the second sum ranges over all edges $e$ contained in $S$. Given $S \Subset \Z^d$ and $\delta \in \Symb^{\partial S}$, we consider:
\begin{align}
Z_S^\Phi := \sum_{w \in \Symb^S} e^{-U_S^{\Phi}(w)},		&	\mbox{ and }	Z^{\Phi,\delta}_S := \sum_{w \in \Symb^S} e^{-U_S^{\Phi}(w\delta)},
\end{align}
where $Z_S^\Phi$ is known as the \emph{partition function} of $S$. Whenever $Z^{\Phi,\delta}_S > 0$, we say that $\delta$ is \emph{$S$-admissible}. For every $S$-admissible $\delta$, define:
\begin{equation}
\Lambda^{\delta}_S(w) := \frac{e^{-U_S^{\Phi}(w\delta)}}{Z^{\Phi,\delta}_S}.
\end{equation}

The collection $\Lambda = \{\Lambda^{\delta}_S\}_{S,\delta}$ is called a \emph{stationary $\Z^d$ Gibbs specification} for the n.n. interaction $\Phi$. Note that each $\Lambda^{\delta}_S$ is a probability measure on $\Symb^S$. For $S' \subseteq S$ and $u \in \Symb^{S'}$, we marginalize as follows:
\begin{equation}
\Lambda^\delta_S(u) = \sum_{v \in \Symb^{S \backslash S'}} \Lambda_S^{\delta}(uv).
\end{equation}
\end{defn}

\begin{defn}
A \emph{(stationary) nearest-neighbour (n.n.) Gibbs measure} for a n.n. interaction $\Phi$ is an MRF $\mu$ on $\Symb^{\Z^d}$ such that, for any finite set $S$ and $\delta \in \Symb^{\partial S}$, if $\mu(\delta) > 0$ then $\delta$ is $S$-admissible and:
\begin{equation}
\mu^{\delta}(w) = \Lambda_S^{\delta}(w),
\end{equation}
for $w \in \Symb^{S}$, where $\{\Lambda^{\delta}_S\}_{S,\delta}$ is the stationary $\Z^d$ Gibbs specification for $\Phi$.
\end{defn}

Every n.n. interaction $\Phi$ has at least one (stationary) n.n. Gibbs measure (special case of a general result in \cite{1-ruelle}). Often there are multiple Gibbs measures for a single $\Phi$. This phenomenon is usually called a \emph{phase transition}. There are several conditions that guarantee uniqueness of Gibbs measures. Some of them are introduced in the next section.

Many classical models can be expressed using this framework (all the following models are \emph{isotropic}, i.e. they have the same constraints in every coordinate direction $\{0,e_i\}$, for $i=1,\dots,d$):
\begin{itemize}
\item Ising model: $\Symb = \{-1,+1\}$, $\Phi(a) = -Ea$, $\Phi(ab) = -Jab$ for constants $E$ (external magnetic field) and $J$ (coupling strength).
\item Potts model: $\Symb = \{1,\dots,q\}$, $q \in \N$, $\Phi(a) = 0$, $\Phi(ab) = -J\delta_{ab}$, where $\delta_{ab}$ is the Kronecker delta.
\item Checkerboard shift: $\Symb = \{1,\dots,k\}$, $k \in \N$, $\Phi(a) = 0$, $\Phi(ab) = 0$ if $a \neq b$, and $\Phi(aa) = \infty$; this can be thought of as the limiting case of the $k$-state Potts model when $J \rightarrow -\infty$.
\item Hard-core model: $\Symb = \{0,1\}$, $\Phi(0) = 0$, $\Phi(1) = \beta$, $\Phi(00) = \Phi(10) = \Phi(01) = 0$, $\Phi(11) = \infty$. The parameter $\lambda = e^{-\beta}$ is called the \emph{activity}.
\end{itemize}

Given a n.n. SFT $X = \mathsf{X}(\mathcal{F})$, a \emph{uniform Gibbs measure} on $X$ is a Gibbs measure corresponding to the n.n. interaction which is $0$ on all n.n. configurations except the forbidden configurations in $\mathcal{F}$ (on which it is $\infty$).

Notice that for a Gibbs measure $\mu$ for $\Phi$, $\supp(\mu) \subseteq \mathsf{X}(\Phi)$. The interaction $\Phi$ is allowed to take the value $\infty$ in order to have Gibbs measures supported on proper subsets of $\Symb^{\Z^d}$. In the following, we introduce a mild property sufficient for having $\supp(\mu) = \mathsf{X}(\Phi)$.

\begin{defn}
\label{Dcond}
An SFT $X$ satisfies the \emph{D-condition} if there exist sequences of finite subsets $\left\{S_n\right\}_n$, $\left\{T_n\right\}_n$ of $\Z^d$ such that $S_n \nearrow \infty$, $S_n \subseteq T_n$, $\frac{|T_n|}{|S_n|} \rightarrow 1$, and for any $u \in \Leng_{U}(X)$, with $U \Subset T_{n}^c$, and $v \in \Leng_{S_n}(X)$, we have that $[u]_X \cap [v]_X \neq \emptyset$. Here, $S_n \nearrow \infty$ means that \emph{$\{S_n\}_n$ tend to infinity in the sense of van Hove}, this is to say, $|S_n| \rightarrow \infty$ and for each $p \in \Z^d$:
\begin{equation}
\lim_{n \rightarrow \infty} \frac{|S_n \triangle (p+ S_n)|}{|S_n|} = 0,
\end{equation}
where $\triangle$ denotes the symmetric difference.
\end{defn}

\begin{prop}[{\cite[Remark 1.14]{1-ruelle}}]
\label{dconsupp}
If $\Phi$ is a n.n. interaction and $\mathsf{X}(\Phi)$ satisfies the D-condition, then for any n.n. Gibbs measure $\mu$ for $\Phi$, we have that $\supp(\mu) = \mathsf{X}(\Phi)$.
\end{prop}

\begin{note}
In \cite[Remark 1.14]{1-ruelle} is considered an assumption even weaker than the D-condition for having $\supp(\mu) = \mathsf{X}(\Phi)$.
\end{note}

We define topological pressure for interactions on a shift space $X$. In order to discuss connections between this definition and topological pressure for functions $f \in C(X)$, we need a mechanism for turning an interaction (which is a function on finite configurations) into a continuous function on the infinite configurations in $X$. We do this as follows for the special case of n.n. interactions $\Phi$. For $x \in \mathsf{X}(\Phi)$, define the (continuous) function:
\begin{equation}
A_\Phi(x) :=  -\Phi\left(x(0)\right) - \sum_{i=1}^{d} \Phi\left(x(\{0,e_i\})\right).
\end{equation}

\begin{defn}
For a n.n. interaction $\Phi$, the \emph{topological pressure of $\Phi$} is defined as:
\begin{equation}
P(\Phi) := \lim_{n \rightarrow \infty} \frac{1}{|\block_n|} \log Z^{\Phi}_{\block_n}.
\end{equation}
\end{defn}

It is well-known \cite[Corollary 3.13]{1-ruelle} that for any sequence such that $S_n \nearrow \infty$,
\begin{equation}
P(\Phi) = \lim_{n \rightarrow \infty} \frac{1}{|S_n|} \log Z^{\Phi}_{S_n}.
\end{equation}

A version of the variational principle (see \cite{1-keller,1-ruelle}) implies that the two definitions given here are equivalent in the sense that $P(\Phi) = P(A_\Phi)$. Notice that considering this, measures of maximal entropy are uniform Gibbs measures. In the case that $\mathsf{X}(\Phi)$ satisfies the D-condition, a measure on $\mathsf{X}(\Phi)$ is an equilibrium state for $A_\Phi$ if it is a Gibbs measure for $\Phi$ (the other direction is always true in the n.n. case \cite[Theorem 3]{1-ruelle}).

\section{Mixing properties}
\label{sec4}

In this section we proceed to introduce some mixing properties of measure-theoretic, combinatorial and topological kind. In general terms, a mixing property tells that, either a measure or the support of it (in most cases an SFT for our purposes), does not have strong long-range correlations. This last aspect will be key for obtaining succinct representations of entropy and pressure, and when developing efficient algorithms for approximating them.

\subsection{Spatial mixing}

The first two definitions are what we call here \emph{spatial mixing} properties, both related to MRFs. In the following, let $f(n):\N \rightarrow \R_{\geq 0}$ be a function such that $\lim_{n \rightarrow \infty} f(n) = 0$.
 
\begin{defn}
An MRF $\mu$ satisfies \emph{weak spatial mixing (WSM) with rate $f(n)$} if for any $W \Subset \Z^d$, $U \subseteq  W$, $u \in \Symb^U$ and $\delta_1,\delta_2 \in \Symb^{\partial W}$ with $\mu(\delta_1), \mu(\delta_2) > 0$,
\begin{equation}
\left| \mu^{\delta_1}(u) - \mu^{\delta_2}(u) \right| \leq \left|U\right|f(\dist(U,\partial W)).
\end{equation}
\end{defn}

Given a set $S \subseteq \Z^d$ and two configurations $s_1,s_2 \in \Symb^S$, the set of positions where they differ is denoted $\Sigma_{S}(s_1,s_2) := \left\{p \in S: s_1(p) \neq s_2(p)\right\}$. We also use the convention that $\dist(S,\emptyset) = \infty$. Considering this, we have the following definition, a priori stronger than WSM.

\begin{defn}
\label{SSMspec}
An MRF $\mu$ satisfies \emph{strong spatial mixing (SSM) with rate $f(n)$} if for any $W \Subset \Z^d$, $U \subseteq W$, $u \in \Symb^U$ and $\delta_1,\delta_2 \in \Symb^{\partial W}$ with $\mu(\delta_1), \mu(\delta_2) > 0$,
\begin{equation}
\left| \mu^{\delta_1}(u) - \mu^{\delta_2}(u) \right| \leq |U|f\left(\dist(U,\Sigma_{\partial W}(\delta_1,\delta_2))\right).
\end{equation}
\end{defn}

We will say that an MRF $\mu$ satisfies WSM (resp. SSM) if it satisfies WSM (resp. SSM) with rate $f(n)$, for some $f(n)$ as before.

\begin{note}
In the literature, it is also common to find the definition of WSM and SSM with the expression $\left| \mu^{\delta_1}(u) - \mu^{\delta_2}(u) \right|$ replaced by the total variation distance of $\mu^{\delta_1}$ and $\mu^{\delta_2}$ on $U$, denoted $\left\| \mu^{\delta_1}|_U - \mu^{\delta_2}|_U \right\|_{TV}$. The definitions here are, a priori, slightly weaker (so the results where SSM is an assumption are also valid for this alternative definition), but sufficient for our purposes.
\end{note}

\begin{lem}[{\cite[Lemma 2.3]{2-marcus}}]
\label{ssmSing}
Let $\mu$ be an MRF such that for any $W \Subset \Z^d$, $q \in W$, $u \in \Symb^{\{q\}}$ and $\delta_1,\delta_2 \in \Symb^{\partial W}$ with $\mu(\delta_1), \mu(\delta_2) > 0$,
\begin{equation}
\left|\mu^{\delta_1}(u) - \mu^{\delta_2}(u) \right| \leq f\left(\dist\left(q,\Sigma_{\partial W}(\delta_1,\delta_2)\right)\right).
\end{equation}

Then, $\mu$ satisfies SSM with rate $f(n)$.
\end{lem}

\begin{rem}
The proof of Lemma \ref{ssmSing} given in \cite{2-marcus} is for MRFs satisfying exponential SSM (see Definition \ref{defnExp}), but its generalization is direct.
\end{rem}

\begin{figure}[ht]
\centering
\includegraphics[scale = 0.95]{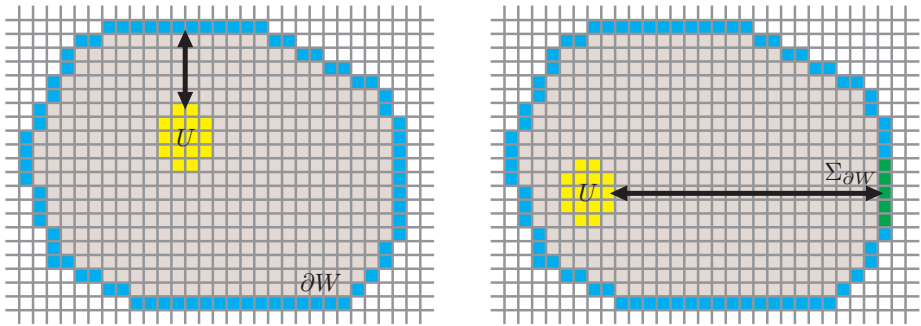}
\caption{The weak and strong spatial mixing properties.}
\label{wsm}
\end{figure}

Notice that SSM implies WSM. If a Gibbs measure $\mu$ for an interaction $\Phi$ satisfies WSM (and $X_\Phi$, the D-condition), then $\mu$ is unique for $\Phi$ \cite{2-weitz}. Also, note that by definition, a necessary condition for $\mu(\delta) > 0$ is $S$-admissibility of $\delta \in \Symb^{\partial S}$. While there may be no finite procedure for determining if a configuration $\delta$ has positive measure, there is a finite procedure for determining if $\delta$ is $S$-admissible. This is an issue that we will have to deal with, especially when developing algorithms (see Section \ref{PressureApp}).

\begin{defn}
\label{defnExp}
An MRF $\mu$ satisfies \emph{exponential WSM} (resp. \emph{exponential SSM}) if it satisfies WSM (resp. SSM) with rate $f(n) = Ce^{-\alpha n}$, for some constants $C,\alpha > 0$.
\end{defn}

There are some well-known models that satisfy exponential SSM:
\begin{itemize}
\item Ising model in $\Z^2$ without external field and $\beta < \beta_c$ (see \cite{1-martinelli}).
\item Anti-ferromagnetic Potts model on $\Z^2$ for $q \geq 6$ (see \cite{2-goldberg}).
\item Checkerboard shift on $\Z^2$ for $k \geq 6$ (see \cite{1-goldberg}).
\item Hard-core model on $\Z^d$ for $\lambda < \lambda_c(2d)$ (see \cite{1-weitz}).
\end{itemize}

There are more general sufficient conditions for having SSM at exponential rate (for instance, see the discussion in \cite{2-marcus}).

\subsection{Measure-theoretic mixing}

A well-known notion of measure-theoretic mixing in ergodic theory (see \cite{1-walters}) is the following one.

\begin{defn}
A shift-invariant measure $\mu$ on a shift space $X$ is \emph{mea\-su\-re-theo\-re\-tic strong mixing} if for any pair of non-empty (disjoint) $U,V \Subset \Z^d$ and for every $u \in \Symb^{U}$, $v \in \Symb^{V}$:
\begin{equation}
\lim_{\|p\| \to \infty} \mu\left([u] \cap \sigma_{-p}([v])\right) =  \mu\left(u\right) \mu\left(v\right).
\end{equation}
\end{defn}

In Section \ref{sec6}, is provided a connection between this and the preceding spatial mixing properties.

\subsection{Topological mixing}

Now we introduce two topological mixing properties that, in this context, will be usually related with the support of an MRF. A very important characteristic of them is that both are conjugacy invariants (see Definition \ref{conj}).

\begin{defn}
\label{topMix}
A shift space $X$ is \emph{topologically mixing} if for any pair of non-empty (disjoint) $U,V \Subset \Z^d$ there exists a separation constant $g(U,V) \in \N$ so that for every $u \in \Symb^{U}$, $v \in \Symb^{V}$ and any $p \in \Z^d$ such that $\dist(U,p+V) \geq g(U,V)$,
\begin{equation}
[u]_X,[v]_X \neq \emptyset \implies [u]_X \cap \sigma_{-p}([v]_X) \neq \emptyset.
\end{equation}
\end{defn}

\begin{defn}
A shift space $X$ is said to be \emph{strongly irreducible with gap $g \in \N$} if for any pair of non-empty (disjoint) finite subsets $U,V \Subset \Z^d$ with separation $\dist(U,V) \geq g$, and for every $u \in \Symb^{U}$, $v \in \Symb^{V}$,
\begin{equation}
[u]_X,[v]_X \neq \emptyset \implies [uv]_X \neq \emptyset.
\end{equation}
\end{defn}

\begin{rem}
Since a shift space is a compact space, it does not make a difference if the shapes of $U$ and $V$ are allowed to be infinite in the definition of strong irreducibility.
\end{rem}

\begin{figure}[ht]
\centering
\includegraphics[scale = 0.45]{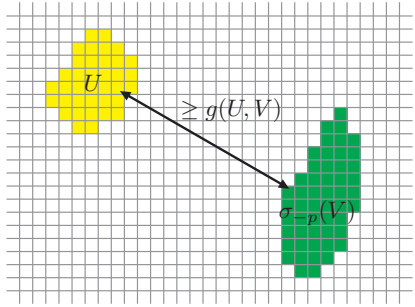}
\caption{The topological mixing property. Strong irreducibility means that $g(U,V)$ can be chosen to be uniform in $U$ and $V$.}
\label{TopMix}
\end{figure}

A first relation between mixing properties and computability of entropy is given by the following theorem.

\begin{thm}[\cite{1-hochman}]
\label{mixComp}
For any strongly irreducible $\Z^d$ SFT $X$, $h(X)$ is computable.
\end{thm}

\subsection{Combinatorial mixing}

The following two properties have in common their local and combinatorial nature related with n.n. constraints. They both also have global scale implications (like, for example, strong irreducibility).

\begin{defn}
Given an alphabet $\Symb$, a list of n.n. forbidden configurations $\mathcal{F}$ and the corresponding n.n. SFT $X = \mathsf{X}(\mathcal{F})$, we say that $a \in \Symb^{\{0\}}$ is a \emph{safe symbol} for $X$ if $\eta{a}$ is locally admissible for every configuration $\eta \in \Symb^{\partial\{0\}}$.
\end{defn}

\begin{exmp}[A n.n. SFT with a safe symbol]
In the support of the $\Z^d$ hard-core model (the n.n. SFT $\mathcal{H}_d$), $0$ is a safe symbol for every $d$ (see \cite{1-gamarnik}).
\end{exmp}

\begin{defn}
\label{ssf}
A n.n. SFT $X$ is \emph{single-site fillable (SSF)} if for some list $\mathcal{F}$ of n.n. forbidden configurations such that $X = \mathsf{X}(\mathcal{F})$, for every $\eta \in \Symb^{\partial\{0\}}$, there exists $a \in \Symb^{\{0\}}$ such that $\eta{a}$ is locally admissible.
\end{defn}

\begin{note}
A n.n. SFT $X$ satisfies SSF if and only if for some forbidden list $\mathcal{F}$ of nearest neighbours that defines $X$, every locally admissible configuration is globally admissible \cite{1-marcus}.
\end{note}

In the definition of SSF above, the symbol $a$ may depend on the configuration $\eta$. Clearly, a n.n. SFT containing a safe symbol satisfies SSF. Also, it is easy to check that a n.n. SFT $X$ that satisfies SSF is strongly irreducible with gap $g = 2$.

\begin{exmp}[A n.n. SFT that satisfies SSF without a safe symbol]
\label{exmpC4}
The $\Z^d$ $k$-checkerboard shift $\mathcal{C}_d(k)$ has no safe symbol for any $k \geq 2$ and $d \geq 1$. However, $\mathcal{C}_d(k)$ satisfies SSF if and only if $k \geq 2d+1$ (see \cite{1-marcus}).
\end{exmp}

\section{Topological strong spatial mixing}
\label{tssmproperties}

Now we introduce a new mixing property, somehow an hybrid between the topological and combinatorial properties from last section. Because of its close relationship with topological Markov fields (see, for example, \cite{2-chandgotia}), we prefer to use the word topological for naming it. This condition will be used to generalize results related with pressure representation and approximation (discussed in Section \ref{PressureRep} and Section \ref{PressureApp}), and also to give a partial characterization of systems that admit measures satisfying SSM.

\begin{defn}
A shift space $X$ satisfies \emph{topological strong spatial mixing with gap $g \in \N$}, if for any $U,V,S \Subset \Z^d$ such that $\dist(U,V) \geq g$, and for every $u \in \Symb^{U}$, $v \in \Symb^{V}$ and $s \in \Symb^{S}$,
\begin{equation}
[us]_X,[sv]_X \neq \emptyset \implies [usv]_X \neq \emptyset.
\end{equation}
\end{defn}

Notice that TSSM implies strong irreducibility (by taking $S = \emptyset$). The difference here is that we allow an arbitrarily close globally admissible configuration on $S$ in between two sufficiently separated globally admissible configurations, provided that each of the two configurations is compatible with the one on $S$, individually. Clearly, TSSM with gap $g$ implies TSSM with gap $g+1$. We will say that a shift space satisfies TSSM if it satisfies TSSM with gap $g$, for some $g \in \N$.

It can be checked that for a n.n. SFT (all implications are strict):
\begin{equation}
\label{mixImpl}
\mbox{Safe symbol} \implies \mbox{SSF} \implies \mbox{TSSM} \implies \mbox{Strong irred.} \implies \mbox{Top. mixing}.
\end{equation}

See Section \ref{examples} for examples that illustrate the differences among some of these conditions.

\subsection{Characterizations and properties of TSSM}

A useful tool when dealing with TSSM is the next lemma.

\begin{lem}
\label{lemSing}
Let $X$ be a shift space and $g \in \N$ such that for every pair of sites $p,q \in \Z^d$ with $\dist(p,q) \geq g$ and $S \Subset \Z^d$, we have that for every $u \in \Symb^{\{p\}}$, $v \in \Symb^{\{q\}}$ and $s \in \Symb^{S}$ with $[us]_X,[sv]_X \neq \emptyset$, then $[usv]_X \neq \emptyset$. Then, $X$ satisfies TSSM with gap $g$.
\end{lem}

\begin{proof}
We proceed by induction. The base case $|U|+|V| = 2$ is given by the hypothesis of the lemma. Now, let's suppose that the property is true for subsets $U,V \Subset \Z^d$ such that $|U|+|V| \leq n$ and let's prove it for the case when $|U|+|V| = n+1$.

Given $U,V,S \Subset \Z^d$ with $S \Subset \Z^d$, $\dist(U,V) \geq g$ and $|U|+|V| = n+1$, and given $u \in \Symb^{U}$, $v \in \Symb^{V}$ and $s \in \Symb^{S}$, we can write $U = \left\{p_1,\dots,p_k\right\}$ and $V = \left\{q_1,\dots,q_m\right\}$, where $|U| = k$, $|V| = m$, $k,m \geq 1$, and $k+m = n+1$. Let's consider $U' = U \backslash \{p_k\}$ and $V' = V \backslash \{q_m\}$, possibly empty sets (but not both empty at the same time, since we can assume that $|U|+|V| > 2$). Similarly, let's consider the restrictions $u' = u(U')$ and $v' = v(V')$. By the induction hypothesis, we have that $[usv']_X, [u'sv]_X \neq \emptyset$ (even in the case $U'$ or $V'$ being empty). Then, if we consider $s' = u'sv'$ on $S' = U' \cup S \cup V'$, we can apply the property for singletons with $u(p_k)$ and $v(q_m)$, and we conclude that $\emptyset \neq [u(p_k)s'v(q_m)]_X = [usv]_X$.
\end{proof}

\begin{rem}
Notice that Lemma \ref{lemSing} states that if we have the TSSM property for singletons, then we have it uniformly (in terms of separation distance) for any pair of finite sets $U$ and $V$.
\end{rem}

\begin{figure}[ht]
\centering
\includegraphics[scale = 0.7]{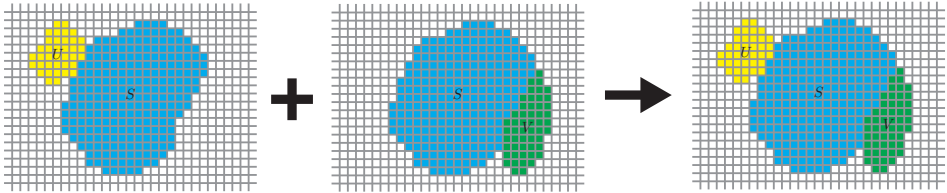}
\caption{The topological strong spatial mixing property.}
\label{TSSM}
\end{figure}

\begin{defn}
Given a shift space $X$ and $W \subseteq \Z^d$, a configuration $w \in \Symb^W$ is called a \emph{first offender for $X$} if $w \notin \Leng(X)$ and $w(S) \in \Leng(X)$, for every $S \subsetneq W$. We define the \emph{set of first offenders of $X$} as:
\begin{equation}
\mathcal{O}(X) := \bigcup_{0 \in W \Subset \Z^d}\left\{w \in \Symb^W \middle\vert w \mbox{\rm ~is a first offender for } X\right\}.
\end{equation}
\end{defn}

\begin{note}
When $d=1$, a similar notion of first offender can be found in \cite[Exercise 1.3.8]{1-lind}, where it is used to characterize a ``minimal'' family $\mathcal{F}$ inducing an SFT $X$.
\end{note}

\begin{prop}
Let $X$ be a shift space. Then $X$ satisfies TSSM iff $\left|\mathcal{O}(X)\right| < \infty$.
\end{prop}

\begin{proof}
First, suppose that $X$ satisfies TSSM with gap $g$, for some $g \in \N$, and take $w \in \mathcal{O}(X)$ with shape $W$ such that $0 \in W$. By contradiction, assume that $\diam(W) \geq g$ and let $p,q \in W$ be such that $\dist(p,q) = \diam(W)$. Then, since $w$ is a first offender, we have that $[w(p)w(S)]_X,[w(S)w(q)]_X \neq \emptyset$, for $S = W \backslash \{p,q\} \subsetneq W$. Since $\dist(p,q) \geq g$, by TSSM, we have that $[w]_X = [w(p)w(S)w(q)]_X \neq \emptyset$, which is a contradiction. Then, $\diam(W) < g$ and, since $0 \in W$, we have that $\left|\mathcal{O}(X)\right| \leq |\Symb|^{(2g+1)^d} < \infty$.

Now, suppose that $\left|\mathcal{O}(X)\right| < \infty$ and take:
\begin{equation}
R = \max_{0 \in W \Subset \Z^d}\left\{\dist(0,q): q \in W, \Symb^W \cap \mathcal{O}(X) \neq \emptyset\right\} < \infty,
\end{equation} which is well defined thanks to the assumption. Consider arbitrary $p,q \in \Z^d$ and $S \Subset \Z^d$, with $\dist(p,q) \geq R+1$, and take $u \in \Symb^{\{p\}}$, $s \in \Symb^S$, $v \in \Symb^{\{q\}}$ such that $[us]_X,[sv]_X \neq \emptyset$. W.l.o.g., by shift-invariance, we can take $p = 0$. Now, by contradiction, assume that $[usv]_X = \emptyset$. Consider a minimal $S' \subseteq S$ such that $[us(S')v]_X = \emptyset$ and $[us(S'')v]_X \neq \emptyset$, for all $S'' \subsetneq S'$ (this includes the case $S' = \emptyset$, where the condition over $S''$ is vacuously true). It is direct to check that $us(S')v$ is a first offender with shape $W = \{0,q\} \cup S'$. Then, since $\dist(0,q) = \dist(p,q) \geq R+1$, we have a contradiction with the definition of $R$. Therefore, thanks to Lemma \ref{lemSing}, $X$ satisfies TSSM with gap $R+1$.
\end{proof}

Notice that $X = \mathsf{X}(\mathcal{O}(X))$. Considering this, we have the following corollary.

\begin{cor}
\label{TSSMSFT}
Let $X$ be a shift space that satisfies TSSM. Then, $X$ is an SFT.
\end{cor}

\begin{note}
If $X \subseteq \Symb^{\Z^d}$ is a shift space that satisfies TSSM with gap $g$, then it can be checked that $X$ is an SFT that can be defined by a family of forbidden configurations $\mathcal{F} \subseteq \Symb^{\romb_g}$.
\end{note}

The next lemma provides another characterization of TSSM for SFTs.

\begin{lem}
\label{rombTSSM}
Let $X = \mathsf{X}(\mathcal{F})$ be an SFT, with $\mathcal{F} \subseteq \Symb^{\romb_N}$ for some $N \in \N$. Then, $X$ satisfies TSSM with gap $g$ if and only if for all $S \subseteq \romb_{g+N-1} \backslash \{0\}$,
\begin{equation}
\label{rombg}
\forall u \in \Symb^{\{0\}}, s \in \Symb^S, v \in \Symb^{\partial_{2N+1}\romb_{g-1} \backslash S}: [us]_X,[sv]_X \neq \emptyset \implies [usv]_X \neq \emptyset.
\end{equation}
\end{lem}

\begin{proof}
Let's prove that if $X$ satisfies Equation \ref{rombg}, then $X$ satisfies TSSM with gap $g$. W.l.o.g., by Lemma \ref{lemSing} and shift-invariance, consider $p,q \in \Z^d$ with $\dist(p,q) \geq g$, $p = 0$, $S \Subset \Z^d$ and configurations $u \in \Symb^{\{p\}}$, $v \in \Symb^{\{q\}}$ and $s \in \Symb^{S}$ such that $[us]_X,[sv]_X \neq \emptyset$. Take $x \in [sv]$ and consider $u' = u$, $v' = x(\partial_{2N+1}\romb_{g-1} \backslash S)$ and $s' = s(\romb_{g+N-1} \cap S)$. Then, $\emptyset \neq [us]_X \subseteq [u's']_X$ and $\emptyset \neq [x(\romb_g)]_X \subseteq [s'v']_X$, so $[u's'v']_X \neq \emptyset$, by Equation \ref{rombg}. Take $y \in [u's'v']$ and notice that $y(\partial_{2N+1}\romb_{g-1}) = x(\partial_{2N+1}\romb_{g-1})$. Then, since $X$ is an SFT defined by a family of configurations $\mathcal{F} \subseteq \Symb^{\romb_N}$, we conclude that $z = y(\romb_{g+N-1})x(\Z^d \backslash \romb_{g+N-1}) \in [usv]_X$, so $[usv]_X \neq \emptyset$ and $X$ satisfies TSSM. The converse is immediate.
\end{proof}

\begin{prop}
\label{TSSMperiod}
Let $X$ be a non-empty $\Z^d$ shift space that satisfies TSSM with gap $g$. Then, $X$ contains a periodic point of period $2g$ in every coordinate direction.
\end{prop}

\begin{proof}
Consider the hypercube $Q = [1,2g]^d$. Notice that $\Z^d = \coprod_{p \in 2g\Z^d} \left(p + Q\right)$. Given ${\ell} \in \{0,1\}^d$, denote $Q({\ell}) = g{\ell} + [1,g]^d \subseteq Q$. Then, $Q = \coprod_{{\ell} \in \{0,1\}^d} Q({\ell})$ and:
\begin{equation}
\Z^d = \coprod_{{\ell} \in \{0,1\}^d} \coprod_{p \in 2g\Z^d} \left(p + Q({\ell})\right) = \coprod_{{\ell} \in \{0,1\}^d} W({\ell}),
\end{equation}
where $W({\ell}) = \coprod_{p \in 2g\Z^d} \left(p + Q({\ell})\right)$. Notice that $\dist(p+Q({\ell}),q+Q({\ell})) \geq g$, for all ${\ell} \in \{0,1\}^d$ and $p,q \in 2g\Z^d$ such that $p \neq q$.

Consider ${0} = {\ell}_0, {\ell}_1, \dots, {\ell}_{2^d-1}$ an arbitrary order in $\{0,1\}^d$. Let $u_0 \in \Leng_{Q({0})}(X)$ and $N \in \N$. By using repeatedly the TSSM property (in particular, strong irreducibility), we can construct a point $x_{0}^{N} \in X$ such that $x_{0}^{N}(2gp + Q({0})) = u_0$, for all $p$ such that $\|p\|_\infty \leq N$. By compactness of $X$, we can take the limit when $N \rightarrow \infty$ and obtain a point $x_0 \in X$ such that $x(2gp + Q({0})) = u_0$, for all $p \in \Z^d$.

Given $0 \leq k < 2^d-1$, suppose that there exists a point $x_k \in X$ and $u_i \in \Leng_{Q({\ell}_i)}(X)$ for $i=0,1,\dots,k$, such that $x_k(2gp + Q({\ell}_i)) = u_i$, for all $i \in \{0,1,\dots,k\}$ and $p \in \Z^d$. Notice that if $k = 2^d-1$, the point $x_{2^d-1}$ is periodic. Then, since we already constructed $x_0$, it suffices to prove that we can construct $x_{k+1}$ from $x_k$. 

Take $u_{k+1} = x_k(Q({\ell}_{k+1}))$. Notice that $u_{k+1} = \sigma_{-2gp}(x_k)(2gp+Q({\ell}_{k+1}))$ and that $\sigma_{-2gp}(x_k)$ has the same property of $x_k$, i.e. $\sigma_{-2gp}(x_k)(2gp+Q({\ell}_i)) = u_i$, for all $i = 0,1,\dots,k$ and $p \in \Z^d$.

Consider an arbitrary enumeration of $\Z^d = \{p_0,p_1,p_2,\dots\}$, with $p_0 = {0}$. Let $x_{k+1}^{0} = x_k$ and suppose that, given $m \in \N$, there is a point $x_{k+1}^{m}$ such that:
\begin{itemize}
\item $x_{k+1}^{m}(2gp_j + Q({\ell}_i)) = u_i$, for all $i \in \{0,1,\dots,k\}$ and $j \in \N$, and
\item $x_{k+1}^{m}(2gp_j + Q({\ell}_{k+1})) = u_{k+1}$, for all $0 \leq j \leq m$.
\end{itemize}

Take the sets $S = \coprod_{r=0}^{k} W({\ell}_r)$, $U = \coprod_{j=0}^{m} 2gp_j+Q({\ell}_{k+1})$ and $V = 2gp_{m+1} + Q({\ell}_{k+1})$. Notice that $\dist(U,V) \geq g$. Then, take $x_{k}(S) \in \Symb^S$, $x_{k+1}^{m}(U) \in \Symb^U$ and $\sigma_{-2gp_{m+1}}(x_k)(V) \in \Symb^V$. We have that $x_{k+1}^{m}(U)x_{k}(S)$ is globally admissible by the hypothesis of the existence of $x_{k+1}^{m}$ and $x_{k}(S)\sigma_{-2gp_{m+1}}(x_k)(V)$ is globally admissible thanks to the observation about $\sigma_{-2gp}(x_k)$. Then, $x_{k+1}^{m}(U)x_{k}(S)\sigma_{-2gp_{m+1}}(x_k)(V)$ is globally admissible, by TSSM. Notice that here $S$ is an infinite set and the TSSM property is for finite sets. This is not a problem since we can consider the finite set $S' = S \cap \block_n$ and take the limit $n \rightarrow \infty$ for obtaining the desired point, by compactness.

\begin{figure}[ht]
\centering
\includegraphics[scale = 0.7]{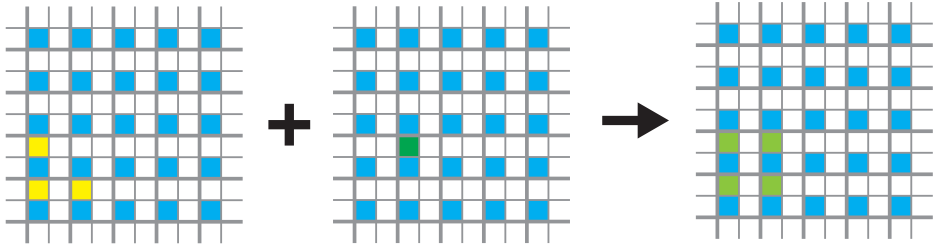}
\caption{Construction of a periodic point using TSSM.}
\label{periodic}
\end{figure}

Now, notice that any extension of $x_{k+1}^{m}(U)x_{k}(S)\sigma_{-2gp_{m+1}}(x_k)(V)$ is a point with the properties of $x_{k+1}^{m+1}$. Taking the limit $m \rightarrow \infty$, we obtain a point with the properties of $x_{k+1}$. Since $k$ was arbitrary, we can iterate the argument until $k = 2^d-1$, for obtaining the point $x_{2^d-1}$ which is periodic of period $2g$ in every canonical direction.
\end{proof}

\begin{note}
It is known that for $d=1,2$ a non-empty strongly irreducible $\Z^d$ SFT contains a periodic point. The case $d=1$ is easy once one knows how to represent an SFT as the space of infinite paths in a directed graph. The case $d=2$ was solved in 2003 by S. Lightwood \cite{1-lightwood}. The case $d \geq 3$ is still an open problem.
\end{note}

\begin{prop}
Let $X$ be a non-empty $\Z^d$ shift space that satisfies TSSM with gap $g$. Then, $X$ contains periodic points of periods $k_1+g,\dots,k_d+g$ in directions $e_1,\dots,e_d$, respectively, for every $k_i \geq g$. Moreover, the set of periodic points is dense in $X$.
\end{prop}

\begin{proof}
Notice that we can modify the proof of Proposition \ref{TSSMperiod}, replacing the hypercube $Q$ by $\prod_{i=1}^{d} [1, k_i + g]$ and the sub-hypercube $Q(0)$, by $\prod_{i=1}^{d} [1, k_i]$. This gives the first part of the statement. Considering this, for checking density of periodic points, notice that in the proof of Proposition \ref{TSSMperiod}, $u_0 \in \Leng_{Q({0})}(X)$ was arbitrary.
\end{proof}

\begin{rem}
In particular, any globally admissible finite configuration of shape $S \subseteq \prod_{i=1}^{d}[1,k_i]$ with $k_i \geq g$, can be embedded in a periodic point of periods $k_1+g,\dots,k_d+g$ in directions $e_1,\dots,e_d$, respectively.
\end{rem}

\begin{lem}
\label{pivot}
Let $X$ be a shift space that satisfies TSSM with gap $g$. Consider $W \Subset \Z^d$ and take $w,w' \in \Leng_{W}(X)$ such that $\Sigma_W(w,w') = \{p_1,\dots,p_k\}$, for some $k \leq |W|$. Then, there exists a sequence $w = w_1, w_2, \dots, w_{k+1} = w' \in \Leng_{W}(X)$ such that $\Sigma_W(w_i,w_{i+1}) \subseteq  \Sigma_W(w,w') \cap \neig_g(p_i)$, for all $1 \leq i \leq k$.
\end{lem}

\begin{proof}
Take $w_1 = w$ and $w' = w'$. By induction, suppose that for some $i \leq k$ we have already constructed a sequence $w_1,\dots,w_i \in \Leng_{W}(X)$ such that:
\begin{align}
\Sigma_W(w_j,w_{j+1}) \subseteq \Sigma_W(w,w') \cap \neig_g(p_j),	&	\mbox{ for all } 1 \leq j < i, \mbox{ and }	\\ \Sigma_W(w_j,w')  \subseteq  \{p_j,\dots,p_k\},						&	\mbox{ for all } 1 \leq j \leq i.
\end{align}
 
The base case $i=1$ is clear. Now, let's extend the sequence to $i+1$. Consider the sets $U_i = \{p_i\}$, $V_i = \Sigma_W(w_i,w') \backslash \neig_g(p_i)$ and $S_i = W \backslash \Sigma_W(w_i,w')$. Take the configurations $u_i \in \Symb^{U_i}$, $v_i \in \Symb^{V_i}$ and $s_i \in \Symb^{S_i}$ defined as $u_i := w'(U_i)$, $v_i := w_i(V_i)$ and $s_i := w_i(S_i) = w'(S_i)$. Since $\dist(U_i,V_i) \geq g$, $\emptyset \neq [w']_X \subseteq [u_is_i]_X$, and $\emptyset \neq [w_i]_X \subseteq [s_iv_i]_X$, by TSSM, we can take $x_i \in [u_is_iv_i]_X$ and consider $w_{i+1} := x_i(W) \in \Leng_{W}(X)$. Then, $\Sigma_V(w_{i+1},w') \subseteq  \{p_{i+1},\dots,p_k\}$ and $\Sigma_W(w_i,w_{i+1}) \subseteq \Sigma_W(w,w') \cap \neig_g(p_i)$, as we wanted. Iterating until $i = k$, we conclude.
\end{proof}

\begin{rem}
Lemma \ref{pivot} is a stronger version of the \emph{generalized pivot property} (see \cite{2-chandgotia}).
\end{rem}

\begin{cor}
\label{boundary}
Let $\mu$ be an MRF such that $\supp(\mu)$ satisfies TSSM with gap $g$. Then, $\mu$ satisfies exponential SSM if and only if for every $W \Subset \Z^d$, $\mu$ satisfies the exponential SSM property restricted to boundaries $\delta_1,\delta_2 \in \Symb^{\partial W}$ with $\mu(\delta_1), \mu(\delta_2) > 0$ and such that $\Sigma_{\partial W}(\delta_1,\delta_2) \subseteq \neig_g(p) \cap \partial W$, for some $p \in \partial W$.
\end{cor}

\begin{proof}
We need to prove that the SSM property holds for boundaries that differ in an arbitrary subset of $\partial W$. By Lemma \ref{ssmSing}, we can restrict our attention to a shape $W \Subset \Z^d$, a site $q \in W$, $u \in \Symb^{\{q\}}$, and boundaries $\delta_1,\delta_2 \in \Symb^{\partial W}$ such that $\mu(\delta_1), \mu(\delta_2) > 0$ and $\dist(q,\Sigma_{\partial W}(\delta_1,\delta_2)) = n$. 

Take an arbitrary $\delta \in \Symb^{\partial W}$ such that $\mu(\delta) > 0$. Define $W_n := W \cap \neig_{n-1}(q)$. By taking averages on $\partial W_n \backslash \partial W$, we have:
\begin{equation}
\label{Eqa}
\mu^{\delta}(u) = \sum_{\gamma: \mu^{\delta}(\gamma) > 0}\mu^{\delta}(u \vert \gamma)\mu^{\delta}(\gamma) = \sum_{\substack{\gamma: \mu^{\delta}(\gamma) > 0 \\ \eta := \gamma\delta(\partial W_n \cap \partial W)}} \mu^{\eta}(u)\mu^{\delta}(\gamma).
\end{equation}
where $\gamma \in \Symb^{\partial W_n \backslash \partial W}$. Now, given arbitrary $\eta,\eta' \in \Symb^{\partial W_n}$ such that $\mu(\eta), \mu(\eta') > 0$, suppose that $\dist(q,\Sigma_{\partial W_n}(\eta,\eta')) = n$ and $\Sigma_{\partial W_n}(\eta,\eta') = \{p_1,\dots,p_k\}$, for some $1 \leq k \leq |\partial W_n|$. Consider the sequence given by Lemma \ref{pivot}, $\eta = \eta_1, \eta_2, \dots, \eta_{k+1} = \eta'$, with $\eta_i \in \Symb^{\partial W_n}$ and $\Sigma_{\partial W_n}(\eta_i,\eta_{i+1}) \subseteq \neig_g(p_i) \cap \partial W_n$, for all $1 \leq i \leq k$. In particular, $\dist(q,\Sigma_{\partial W_n}(\eta_i,\eta_{i+1})) \geq n-g$. Fix $\epsilon \in (0,\alpha)$ and define $\alpha_\epsilon := \alpha-\epsilon$. Then, we can always find a constant $C_\epsilon \geq C$ such that:
\begin{align}
\label{Eqb}
\left| \mu^{\eta}(u) - \mu^{\eta'}(u) \right|	&	\leq \sum_{i=1}^{k} \left| \mu^{\eta_i}(u) - \mu^{\eta_{i+1}}(u) \right|	\\
								&	\leq (2n+1)^dCe^{-\alpha (n-g)} \leq C_\epsilon e^{-{\alpha_\epsilon}n},
\end{align}
because $k \leq |\partial W_n| \leq (2n+1)^d$. Notice that the last inequality holds for sufficiently large $n$ when $C_\epsilon = Ce^g$, but we can always adjust $C_\epsilon$ to obtain the bound for all $n$. Then, combining Equation \ref{Eqa} and Equation \ref{Eqb}, we have that, for $i=1,2$:
\begin{equation}
\left|\mu^{\delta_i}(u) - \mu^{\eta^i}(u)\right| \leq C_\epsilon e^{-{\alpha_\epsilon}n},
\end{equation}
where $\eta^i := \gamma^i\delta_i(\partial W_n \cap \partial W)$, for an arbitrary $\gamma^i \in \Symb^{\partial W_n \backslash \partial W}$ such that $\mu^{\delta_i}(u \vert \gamma^i) > 0$. Therefore, we conclude that:
\begin{align}
\left|\mu^{\delta_1}(u) - \mu^{\delta_2}(u)\right|	&	\leq \left|\mu^{\delta_1}(u) - \mu^{\eta^1}(u)\right| + \left|\mu^{\eta^1}(u) - \mu^{\eta^2}(u)\right|	\\
									&	\phantom{\leq} + \left|\mu^{\eta^2}(u) - \mu^{\delta_2}(u)\right|		\nonumber \\
									&	\leq  C_{\epsilon}e^{-{\alpha_\epsilon}n} + C_{\epsilon}e^{-{\alpha_\epsilon}n} + C_{\epsilon}e^{-{\alpha_\epsilon}n} = 3C_{\epsilon}e^{-{\alpha_\epsilon}n}.
\end{align}
\end{proof}

\begin{rem}
The sequence $\eta_1, \eta_2, \dots, \eta_{k+1}$ from the proof of Corollary \ref{boundary} is called a \emph{sequence of interpolating configurations} \cite[Definition 2.4]{2-martinelli}. In the case without hard constraints (i.e. a full shift), this sequence can always be chosen such that $\Sigma_{\partial W}(\eta_i,\eta_{i+1}) = p_i$, for some $p_i \in \partial W$. In fact, it is common (see \cite{1-martinelli,2-martinelli,1-weitz}) to find as alternative definitions of SSM, boundaries that differ only on a single site. However, when dealing with hard constraints, a definition restricted to boundaries differing on a single site is not necessarily enough for being equivalent to Definition \ref{SSMspec}. Corollary \ref{boundary} gives a similar equivalence, but restricted to boundaries that differ on a neighbourhood of constant size.
\end{rem}

\begin{prop}
\label{ssftssm}
If a n.n. SFT $X$ satisfies SSF, then it satisfies TSSM with gap $g = 2$.
\end{prop}

\begin{proof}
Since $X$ satisfies SSF, every locally admissible configuration is globally admissible. If we take $g = 2$, for all disjoint sets $U,S,V \Subset \Z^d$ such that $\dist(U,V) \geq g$ and for every $u \in \Symb^U$, $s \in \Symb^S$ and $v \in \Symb^U$, if $[us]_X, [sv]_X \neq \emptyset$, in particular we have that $us$ and $sv$ are locally admissible. Since $\dist(U,V) \geq g = 2$, $usv$ must be locally admissible, too. Then, by SSF, $usv$ is globally admissible and, therefore, $[usv]_X \neq \emptyset$.
\end{proof}

It is well known that in the one-dimensional SFT case the mixing hierarchy collapse, i.e. topologically mixing, strongly irreducible and other intermediate properties, such as block gluing and uniform filling, are all equivalent (for example, see \cite{1-boyle}). In the nearest-neighbour case, we extend this to TSSM.

\begin{prop}
\label{1dTSSM}
A $\Z$ n.n. SFT $X$ satisfies TSSM if and only if it is topologically mixing.
\end{prop}

\begin{proof}
We prove that if $X$ is topologically mixing, then it satisfies TSSM. The other direction is obvious.

It is known that a topologically mixing $\Z$ n.n. SFT $X$ is strongly irreducible with gap $g = g(0,0)$, where $g(0,0)$ is the gap according to Definition \ref{topMix}. Consider arbitrary $p,q \in \Z$ and $S \Subset \Z$ such that $\dist(p,q) \geq g$. W.l.o.g., by shift-invariance, assume that $p=0 < q$. Take $u \in \Symb^{\{p\}}$, $v \in \Symb^{\{q\}}$ and $s \in \Symb^{S}$ with $[us]_X, [sv]_X \neq \emptyset$. 

First, consider the interval $(p,q)$ and suppose that $S \cap (p,q) = \emptyset$. By strong irreducibility, there is $w \in \Leng_{(p,q)}(X)$ such that $[uwv]_X \neq \emptyset$. Consider $x \in [us]_X$, $y \in [sv]_X$ and $z \in [uwv]_X$. Then, $x((-\infty,p])z((p,q))y([q,\infty)) \in [usv]_X$, so $[usv]_X \neq \emptyset$. Now, suppose that  $S \cap (p,q) \neq \emptyset$. Take $r \in S \cap (p,q)$ and $x \in [us]$, $y \in [sv]$. Then, $x((-\infty,r])y((r,\infty)) \in [usv]_X$, and $[usv]_X \neq \emptyset$. Finally, we conclude by Lemma \ref{lemSing}.
\end{proof}

As it was mentioned before, topologically mixing and strong irreducibility are stable under conjugacy. However, as most properties which are natural for MRFs (e.g. safe symbol, SSF, etc.), TSSM is not a conjugacy invariant. This is illustrated in the next example.

\begin{exmp}
Given $\Symb = \{0,1,2\}$ and the family of forbidden configurations $\mathcal{F} = \left\{00, 102, 201\right\}$, we can consider the one-dimensional SFT $X = \mathsf{X}(\mathcal{F})$ (not nearest-neighbor). Notice that any point of $X$ can be understood as a sequence of $0$s, $1$s and $2$s, such that in between every pair of consecutive $0$s (which are never adjacent), there is a configuration of $1$s and $2$s freely concatenated with only one restriction: If there is a $0$ in between two configurations of $1$s and $2$s, then the last letter of the configuration at the left of the $0$ is the same as the first of the configuration at the right of it.

It can be checked that $X$ is strongly irreducible with gap $g = 3$. Given two arbitrary configurations $u,v \in \Leng(X)$, we can always extend both of them in order to assume that $u$ has shape $(-\infty,0]$ and $v$ has shape $[p,\infty)$, for some $p \in \Z$. Then, there are four main cases:
\begin{itemize}
\item If $u = u'10$ and $v = 01v'$, then $u1v \in X$.
\item If $u = u'10$ and $v = 02v'$, then $u12v \in X$ (this case needs the biggest gap).
\item If $u = u'10$ and $v = 1v'$, then $uv \in X$.
\item If $u = u'20$ and $v = 1v'$, then $u2v \in X$.
\item If $u = u'1$ or $u = u'2$, and $v = 1v'$ or $v = 2v'$, then $uv \in X$.
\end{itemize}

The remaining cases are analogous, so $X$ is strongly irreducible. However, $X$ is not TSSM. In fact, given $g \in \N$, consider $S = \left\{p \in \Z: 0 < p < 2g, p \mbox{ odd}\right\}$ and the configurations $s = 0^S$, $u = 1^{\{0\}}$ and $v = 2^{\{2g\}}$. Then, $[us]_X, [sv]_X \neq \emptyset$, because $us$ can be extended with $1$s in $\Z \backslash (S \cup \{0\})$ and $sv$ can be extended with $2$s in $\Z \backslash (S \cup \{2g\})$. However, $[usv]_X = \emptyset$, since the $1$ in $u$ forces any point in $[us]_X$ to have value $1$ in $(0,2g) \backslash S$ and the $2$ in $v$ forces any point in $[sv]_X$ to have value $2$ in $(0,2g) \backslash S$. Therefore, since $g$ was arbitrary, $X$ is not TSSM for any gap $g$.

Now, if we define $Y := \beta_1(X)$, where $\beta_1$ is the higher block code with $N=1$ (see Example \ref{blockcode}), then $Y$ is a $\Z$ n.n. SFT conjugate to $X$ ($X \cong Y$), and therefore strongly irreducible (which is a conjugacy invariant). Then, by Proposition \ref{1dTSSM}, we have that $Y$ is TSSM, while $X$ is not.

This example can be extended to any dimension $d$ by considering the constraints $\mathcal{F}$ in only one canonical direction. In other words, TSSM is not a conjugacy invariant for any $d$.
\end{exmp}

\subsection{TSSM and uniform bounds of conditional probabilities}
\label{subsect3}

Now we show how TSSM is closely related with bounds on conditional probabilities of measures satisfying SSM. First, some definitions.

\begin{defn}
\label{Pmu}
Let $\mu$ be a shift-invariant measure on $\Z^d$ and denote $X = \supp(\mu)$. For a set $S \Subset \Z^d \backslash \{0\}$, define $p_{\mu,S}:X \to [0,1]$ to be:
\begin{equation}
p_{\mu,S}(x) := \mu\left(x(0) \middle\vert x(S)\right).
\end{equation}

Notice that $p_{\mu,S}(x)$ is a value that depends only on $x(S \cup \{0\})$. Given this, we define:
\begin{equation}
\Cmu := \inf\left\{p_{\mu,S}(x): x \in \supp(\mu), S \Subset \Z^d \backslash \{0\}\right\}.
\end{equation}
\end{defn}

This and similar uniform bounds were introduced in \cite{1-marcus} for obtaining convergence results and control over certain functions related with topological pressure representation (see Section \ref{PressureRep}). In the same work, it is proven that $\Cmu > 0$ for any n.n. Gibbs measures $\mu$ whose support satisfies SSF. Here we extend this result to MRFs whose support satisfies TSSM. Before that, for an MRF $\mu$ and $T \Subset \Z^d$, we define $D_\mu(T)$ to be:
\begin{equation}
D_\mu(T) := \min_{W \subseteq T} \min_{\substack{\delta \in \Symb^{\partial W} \\ \mu(\delta) > 0}} \min_{\substack{w \in \Symb^{W} \\ \mu(w \vert \delta) > 0}} \mu(w \vert \delta).
\end{equation}

Notice that $D_\mu(T) > 0$, for all $T \Subset \Z^d$.

\begin{prop}
\label{TSSM-Cmu}
Let $\mu$ be an MRF whose support $\supp(\mu)$ satisfies TSSM. Then, $\Cmu > 0$.
\end{prop}

\begin{proof}
Let's denote $X = \supp(\mu)$, and consider $x \in X$ and $S \Subset \Z^d \backslash \{0\}$. Let $K$ be the connected component of $\Z^d \backslash S$ containing $0$ and let $g$ be the gap given by the TSSM property. Define $K_g := K \cap \block_{g-1}$ and $V := \partial K_g \backslash S$. Notice that $V \subseteq K \cap \partial \block_{g-1}$, and $|\partial \block_{g-1}| = 2d(2g+1)^{d-1}$.

First, assume that $V =\emptyset$. If this is the case, then $\partial K_g \subseteq S$. Therefore, by the MRF property:
\begin{equation}
p_{\mu,S}(x) = \mu(x(0) \vert x(S)) =  \mu(x(0) \vert x(\partial K_g))  \geq \mu(x(K_g) \vert x(\partial K_g)) \geq D_\mu(\block_g).
\end{equation}

On the other hand, suppose that $V \neq \emptyset$. By a counting argument, there must exist $v \in \Symb^{V}$ such that:
\begin{equation}
\mu\left(v \middle\vert x(S)\right) \geq |\Symb|^{-|V|} \geq |\Symb|^{-|\partial \block_{g-1}|} = |\Symb|^{-2d(2g+1)^{d-1}}.
\end{equation}

In particular, $vx(S) \in \Leng(X)$. Since $x(S)x(0) \in \Leng(X)$ and $\dist(0,V) \geq g$, by TSSM, we conclude that $vx(S)x(0) \in \Leng(X)$. Now, take $y \in [vx(S)x(0)]_X$. Then, by the MRF property, it follows that:
\begin{align}
p_{\mu,S}(x)	&	=	\mu\left(y(0)	\middle\vert y(S)\right)	\\
			&	\geq	\mu\left(y(K_g) 	\middle\vert y(S)\right)	\\
			&	\geq	\mu\left(y(K_g)	\middle\vert y(S)y(V)\right)\mu\left(y(V) \middle\vert y(S)\right)	\\
			&	=	\mu\left(y(K_g)	\middle\vert y(\partial K_g \cap S)y(V)\right)\mu\left(v \middle\vert x(S)\right)	\\
			&	=	\mu\left(y(K_g)	\middle\vert y(\partial K_g)\right)\mu\left(v \middle\vert x(S)\right)	\\
			&	\geq 	D_\mu(\block_g)|\Symb|^{-2d(2g+1)^{d-1}}.
\end{align}

Therefore, in both cases we have that:
\begin{equation}
p_{\mu,S}(x) \geq D_\mu(\block_g)|\Symb|^{-2d(2g+1)^{d-1}}.
\end{equation}

Since this lower bound is positive and independent of $x$ and $S$, taking the infimum over $S$, we conclude that $\Cmu > 0$.
\end{proof}

An interesting fact is that the converse also holds, at least when $\mu$ satisfies SSM.

\begin{prop}
\label{pSSM-pCmu}
Let $\mu$ be an MRF that satisfies SSM such that $\Cmu > 0$. Then, $\supp(\mu)$ satisfies TSSM.
\end{prop}

\begin{proof}
Let's denote $X = \supp(\mu)$ and assume that $\mu$ satisfies SSM with rate $f(n)$, for some $f(n)$ such that $\lim_{n \rightarrow \infty} f(n) = 0$. Take $n_0 \in \N$ such that $f(n) <  \Cmu$, for all $n \geq n_0$. 

Consider $p,q \in \Z^d$ and $S \subseteq \Z^d$ with $\dist(p,q) = n \geq n_0$, and configurations $u \in \Symb^{\{q\}}$, $v \in \Symb^{\{p\}}$, $s \in \Symb^S$, as in Lemma \ref{lemSing}. W.l.o.g., by shift-invariance, we can assume $p = 0$.

By contradiction, suppose that $[us]_X, [sv]_X \neq \emptyset$, but $[usv]_X = \emptyset$. Then, we have that $\mu\left(v \middle\vert su\right) = 0$. However, since $\mu\left(v \middle\vert s\right) > 0$, by taking an average over configurations on $\{q\}$, there must exist $\tilde{u} \in \Symb^{\{q\}}$ such that $\mu\left(v \middle\vert s\tilde{u}\right) \geq \mu\left(v \middle\vert s\right) = p_{\mu,S}(x) \geq \Cmu$, where $x$ is any element from $[sv]_X$.

Notice that there must exist a path $\camino$ from $0$ to $q$ contained in $\Z^d \backslash S$. If not, by the MRF property, $0 = \mu\left(v \middle\vert su\right) = \mu\left(v \middle\vert s\right) > 0$, which is a contradiction.

Now, take $N$ sufficiently large so $\left(S \cup \camino\right) \subseteq \block_{N}$. Given the set $\block_N \backslash \left(S \cup \{q\}\right)$, consider the connected component $K$ that contains $\{0\}$. Notice that $q$ must belong to $\partial K$. Next, by taking averages over configurations in $\partial K$ and due to the MRF property, there must exist $\delta_1,\delta_2 \in \partial K$ such that:
\begin{align}
 0 = \mu\left(v \middle\vert su\right) \geq \mu^{\delta_1}\left(v\right),	&	\mbox{ and } \mu\left(v \middle\vert s\tilde{u}\right) \leq \mu^{\delta_2}\left(v\right).
 \end{align}
 
Then, since $\Cmu \leq \mu\left(v \middle\vert s\tilde{u}\right)$ and $f(\dist(0,q)) \leq f(n_0)$:
\begin{align}
\Cmu \leq	\left|\mu\left(v \middle\vert su\right) - \mu\left(v \middle\vert s\tilde{u}\right)\right|	 \leq	\left|\mu^{\delta_1}\left(v\right) - \mu^{\delta_2}\left(v\right)\right| \leq f(n_0) < \Cmu,
\end{align}
which is a contradiction. Therefore, $[usv]_X \neq \emptyset$ and, by Lemma \ref{lemSing}, we have that $\supp(\mu)$ satisfies TSSM with gap $g = n_0$.
\end{proof}

\begin{cor}
\label{SSMcoroll}
Let $\mu$ be an MRF that satisfies SSM. Then, $\Cmu > 0$ if and only if $\supp(\mu)$ satisfies TSSM.
\end{cor}

\section{Connections between mixing properties}
\label{sec6}

In this section we establish some connections between boundary and combinatorial/topological mixing properties. In particular, we show how TSSM is a property that arises naturally when we have an MRF satisfying SSM, at least when the decay rate is high enough.

\begin{prop}
\label{WSM-mix}
If an MRF $\mu$ satisfies WSM, then $\mu$ is measure-theoretic strong mixing. In particular, $\mu$ is ergodic and $\supp(\mu)$ is topologically mixing.
\end{prop}

\begin{proof}
Consider $U,V \Subset \Z^d$, $u \in \Symb^U$, $v \in \Symb^V$ and, w.l.o.g., $\mu(u),\mu(v)  > 0$. Given $\epsilon > 0$ and the rate $f(n)$ of WSM, take $n_0 \in \N$ such that $f(n) \leq \frac{\epsilon}{|U|}$, for all $n \geq n_0$. Given $p \in \Z^d$ such that $\dist(U,p+V) \geq 2n_0$, denote $v'$ the translated version of $v$, from $V$ to $p+V$. Then, 
\begin{equation}
\mu(uv') =	\mu\left(v'\right)\mu\left(u \middle\vert v'\right) =	\mu\left(v'\right) \sum_{\delta} \mu\left(u \delta\middle\vert v'\right),
\end{equation}
where the sum ranges over all boundary configurations $\delta \in \Symb^{\partial \neig_{n_0}(U)}$ such that $\mu(\delta | v') > 0$. By shift-invariance, $\mu(v') = \mu(v)$, so (by the MRF property):
\begin{equation}
\mu(uv') =	\mu\left(v\right) \sum_{\delta}\mu\left(u \middle\vert \delta v'\right)\mu\left(\delta\middle\vert v'\right) =	\mu\left(v\right) \sum_{\delta}\mu\left(u \middle\vert \delta\right)\mu\left(\delta\middle\vert v'\right).
\end{equation}

Now, since $\sum_{\delta}\mu\left(\delta\middle\vert v'\right) = 1$, we have that:
\begin{equation}
\mu\left(v\right)\mu\left(u \middle\vert \underline{\delta}\right)	\leq	\mu(uv') \leq	\mu\left(v\right)\mu\left(u \middle\vert \overline{\delta}\right),
\end{equation}
where $\underline{\delta}, \overline{\delta} \in \Symb^{\partial \neig_{n_0}(U)}$ are boundary configurations such that $\mu\left(u \middle\vert \underline{\delta}\right) \leq \mu\left(u \middle\vert \delta\right) \leq \mu\left(u \middle\vert \overline{\delta}\right)$, for every $\delta$. By WSM, and since $\mu(u) = \sum_{\delta} \mu\left(u \delta\middle\vert \delta\right)\mu\left(\delta\right)$, we have that $\left|\mu\left(u\right) - \mu\left(u \middle\vert \delta\right)\right| \leq |U|f(n)$, for every $\delta$. Therefore,
\begin{equation}
\mu\left(v\right)\mu\left(u\right) - |U|f(n) \leq	\mu(uv') \leq	\mu\left(v\right)\mu\left(u\right) + |U|f(n).
\end{equation}

Then, since $f(n) \leq \frac{\epsilon}{|U|}$, we have that $\left| \mu(uv') - \mu\left(v\right)\mu\left(u\right) \right| \leq |U|f(n) \leq \epsilon$, we conclude.
\end{proof}

In contrast with the preceding result involving WSM, we have the following one with the SSM assumption.

\begin{thm}
\label{rate}
Let $\mu$ be a $\Z^2$ MRF that satisfies exponential SSM with rate $f(n) = Ce^{-\alpha{n}}$, where $\alpha > 4\log|\Symb|$. Then, $\supp(\mu)$ satisfies TSSM.
\end{thm}

\begin{proof}
We will prove that $\mathrm{c}_\mu > 0$ and then conclude thanks to Corollary \ref{SSMcoroll}. Let's denote $X = \supp(\mu)$, and consider $x \in X$ and $S \Subset \Z^d \backslash \{0\}$. Our goal is to bound $p_{\mu,S}(x)$ away from zero, uniformly in $x$ and $S$. Let $K$ be the connected component of $\Z^d \backslash S$ containing $0$. Given $n \in \N$ such that:
\begin{equation}
\alpha - 4\log|\Symb| > \frac{1}{n}\log(4Cn),
\end{equation}
take the $n$-rhomboid $\romb_n$, and define $K_n := K \cap \romb_{n-1}$ and $V := \partial K_n \backslash S$. Similarly to the proof of Proposition \ref{TSSM-Cmu}, notice that $V \subseteq K \cap \partial \romb_{n-1}$, and $|\partial \romb_{n-1}| = 4n$ (here we consider $n$-rhomboids instead of $n$-blocks for reasons explained later). If $V =\emptyset$, then $\partial K_n \subseteq S$. Therefore,
\begin{equation}
p_{\mu,S}(x) = \mu(x(0) \vert x(S)) = \mu(x(0) \vert x(\partial K_n)) \geq \mu(x(K_n) \vert x(\partial K_n)) \geq D_\mu(\romb_n).
\end{equation}

Now, as in the proof of Proposition \ref{TSSM-Cmu}, let's suppose that $V \neq \emptyset$. By a counting argument, there must exist $v \in \Symb^{V}$ such that $\mu\left(v \middle\vert x(S)\right) \geq |\Symb|^{-|V|}$ and, in particular, $vx(S) \in \Leng(X)$.

By contradiction, let's suppose that $vx(S)x(0) \notin \Leng(X)$. Then, $\mu\left(v \middle\vert x(S)x(0)\right) = 0$. On the other hand, since $\mu\left(v \middle\vert x(S)\right) \geq |\Symb|^{-|V|}$, there must exist $u \in \Symb^{\{0\}}$ such that $\mu\left(v \middle\vert x(S)u\right) \geq |\Symb|^{-|V|}$ (by taking averages over configurations on $\{0\}$).

Now, let $T_{n} := (K \cap \romb_{2n-1}) \backslash \{0\}$, $F := \partial T_{n} \backslash (S \cup \{0\})$, and $H := \partial T_{n} \backslash F$. Notice that $0 \in H$. Also, $V \subseteq T_{n}$, so $F \sqcup H$ surrounds $V$. By taking averages over configurations in $F$, it is always possible to find $\eta_1,\eta_2 \in \Symb^{F}$ such that $\eta_1x(S)x(0), \eta_2x(S)u \in \Leng(X)$, and:
\begin{align}
 0 = \mu\left(v \middle\vert x(S)x(0)\right) \geq \mu\left(v \middle\vert \eta_1x(S)x(0)\right),	&	\mbox{ and } \mu\left(v \middle\vert x(S)u\right) \leq \mu\left(v \middle\vert \eta_2x(S)u\right).
 \end{align}

Take $\delta_1,\delta_2 \in \Symb^{\partial T_n}$ with $\delta_1 = \eta_1x(H \backslash \{0\})x(0)$ and $\delta_2 = \eta_2x(H \backslash \{0\})u$ and notice that $\dist(V,\Sigma_{\partial T_n}(\delta_1,\delta_2)) \leq \dist(V,\{0\} \cup F) = n$ (notice that it could be the case that $u = x(0)$ and $\eta_1 = \eta_2$). Then, we have that:
\begin{align}
|\Symb|^{-|V|}	&	\leq	\left|\mu\left(v \middle\vert x(S)x(0)\right) - \mu\left(v \middle\vert x(S)u\right)\right|	\\
			&	\leq	\left|\mu\left(v \middle\vert \eta_1x(S)x(0)\right) - \mu\left(v \middle\vert \eta_2x(S)u \right)\right|	\\
			&	=	\left|\mu^{\delta_1}\left(v\right) - \mu^{\delta_2}\left(v\right)\right|		\\
			&	\leq	|V|Ce^{-\alpha{n}},
\end{align}
by the MRF and SSM properties. Since $V \subseteq \partial \romb_{n-1}$, then $|V| \leq |\partial \romb_{n-1}|$ and:
\begin{equation}
|\Symb|^{-|\partial \romb_{n-1}|}	\leq	|\Symb|^{-|V|}	\leq	|V|Ce^{-\alpha{n}}	\leq	|\partial \romb_{n-1}|Ce^{-\alpha{n}}.
\end{equation}

\begin{figure}[ht]
\centering
\includegraphics[scale = 0.7]{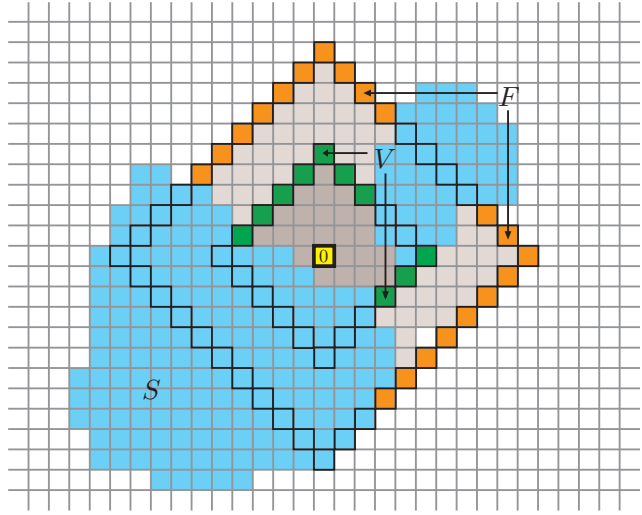}
\caption{Representation of  $\romb_n$,  $\romb_{2n}$ and the subsets $V$ (green), $F$ (orange) and $S$ (light blue) in the proof of Theorem \ref{rate}.}
\label{highrate}
\end{figure}

By taking logarithms, $-4n\log|\Symb|	\leq \log(4n) + \log C -\alpha{n}$, so:
\begin{equation}
\alpha \leq \frac{1}{n}\log(4Cn) + 4\log|\Symb| = 4\log|\Symb| + o(1),
\end{equation}
which is a contradiction with the fact that $\alpha > 4\log|\Symb|$ for $n$ sufficiently large (notice that the difference between $\alpha$ and $4\log|\Symb|$ determines the size of $|V|$ and its distance to $0$). Then, we conclude that $vx(S)x(0) \in \Leng(X)$. Therefore, by considering $y \in [vx(S)x(0)]_X$ and repeating the argument in the proof of Proposition \ref{TSSM-Cmu}, we have that:
\begin{equation}
p_{\mu,S}(x) \geq D_\mu(\romb_n)|\Symb|^{-|\partial \romb_{n-1}|} = D_\mu(\romb_n)|\Symb|^{-4n}.
\end{equation}

Since this lower bound is positive and independent of $x$ and $S$, taking the infimum over $S$, we have that $\Cmu \geq D_\mu(\romb_n)|\Symb|^{-4n} > 0$ and, by Corollary \ref{SSMcoroll}, we conclude that $\supp(\mu)$ exhibits TSSM.
\end{proof}

\begin{rem}
Recall that TSSM implies strong irreducibility, so in view of the preceding result SSM with high exponential rate implies strong irreducibility. In general, it is not known whether SSM implies strong irreducibility.
\end{rem}

\begin{note}
In general, if $\mu$ were a $\Z^d$ MRF satisfying SSM with rate $f(n) = Ce^{-\alpha{n^{d-1}}}$, we could modify the previous proof to conclude that $\supp(\mu)$ exhibits TSSM for sufficiently large $\alpha$. The reason why exponential SSM is not enough in this proof for an arbitrary $d$, is that only in $\Z^2$ the boundary of balls grows linearly with the radius. This is also related with the choice of $\romb_n$ over $\block_n$ in the previous proof, since $|\partial \romb_n| \leq |\partial \block_n|$ and this optimizes the bound for $\alpha$. In this sense, the previous proof should work with any lattice where the boundary of balls grows linearly with the radius (probably under a change of the bound for the rate $\alpha$).
\end{note}

\section{Examples}
\label{examples}

In this section we exhibit examples of n.n. SFTs which illustrate some if the mixing properties discussed in this work.

\subsection{A n.n. SFT that satisfies strong irreducibility, but not TSSM}

Clearly, the SSF property implies strong irreducibility (a way to see this is through Proposition \ref{ssftssm}). As it is mentioned in Example \ref{exmpC4}, $\mathcal{C}_2(k)$ (the $\Z^2$ $k$-checkerboard) satisfies SSF if and only if $k \geq 5$. For the n.n. SFT $\mathcal{C}_2(4)$, given $\eta \in \Symb(4)^{\partial\{0\}}$ (where $\Symb(4) := \{1,2,3,4\}$) defined by $\eta(e_1) = 1$, $\eta(e_2) = 2$, $\eta(-e_1) = 3$, and $\eta(-e_2) = 4$, there is no $a \in \Symb(4)^{\{0\}}$ such that $\eta a$ remains locally admissible, so $\mathcal{C}_2(4)$ does not satisfy SSF.  However, inspired in the SSF property, we have the following definition.

\begin{defn}
Given $N \in \N$, a n.n. SFT X satisfies $N$-fillability if, for every locally admissible configuration $\delta \in \Symb^T$, with $T \subseteq \Z^d \backslash [1,N]^d$, there exists $w \in \Symb^{[1,N]^d}$ such that $w\delta$ is locally admissible.
\end{defn}

\begin{rem}
In the previous definition, since $X$ is a n.n. SFT, it is equivalent to consider $\delta$ to have shape $T \subseteq \partial[1,N]^d$. In this sense, notice that 1-fillability coincides with the notion of SSF (which only considers locally admissible configurations on $\partial\{0\}$).
\end{rem}

\begin{lem}
The n.n. SFT $\mathcal{C}_2(4)$ satisfies $2$-fillability.
\end{lem}

\begin{proof}
Consider an arbitrary locally admissible configuration $\delta \in \Symb(4)^T$, with $T \subseteq \Z^d \backslash [1,2]^2$. We want to check if there is $w \in \Symb(4)^{[1,2]^2}$ such that $w\delta$ remains locally admissible. W.l.o.g., we can assume that $T = \partial [1,2]^2$, which is the worst case. Given a locally admissible boundary $\delta \in \Symb(4)^{\partial [1,2]^2}$ and $p \in [1,2]^2$, let's denote by $A^{\delta}_p$ the set of values $a \in \Symb(4)^{\{p\}}$ such that $a\delta$ remains locally admissible. Notice that $|A^{\delta}_p| \geq 2$, for every $p \in [1,2]^2$ and for every $\delta$. W.l.o.g., assume that $|A^{\delta}_p| = 2$, $A^{\delta}_{(1,1)} = \{1,2\}$, and consider $w \in \Symb(4)^{[1,2]^2}$ to be defined.

First, suppose that $A^{\delta}_{(1,1)} \cap A^{\delta}_{(2,2)} \neq \emptyset$ or $A^{\delta}_{(2,1)} \cap A^{\delta}_{(1,2)} \neq \emptyset$. By the symmetries of $[1,2]^2$ and the constraints, we may assume that $1 \in A^{\delta}_{(1,1)} \cap A^{\delta}_{(2,2)}$ and take $w(1,1) = w(2,2) = 1$, $w(2,1) \in A^{\delta}_{(2,1)} \backslash \{1\}$ and $w(1,2) \in A^{\delta}_{(1,2)} \backslash \{1\}$. It is easy to check that $w\delta$ is locally admissible.

On the other hand, if $A^{\delta}_{(1,1)} \cap A^{\delta}_{(2,2)} = \emptyset$ and $A^{\delta}_{(2,1)} \cap A^{\delta}_{(1,2)} = \emptyset$, we have that $A^{\delta}_{(0,0)} = \{1,2\}$ and $A^{\delta}_{(1,1)} = \{3,4\}$. We consider two cases based on whether a diagonal and off-diagonal coincide or intersect in exactly one element:
\begin{itemize}
\item If $A^{\delta}_{(2,1)} = \{1,2\}$ and $A^{\delta}_{(1,2)} = \{3,4\}$, we can take $w(1,1) = 1$, $w(2,1) = 2$, $w(1,2) = 3$, $w(2,2) = 4$.
\item If $A^{\delta}_{(2,1)} = \{1,3\}$ and $A^{\delta}_{(1,2)} = \{2,4\}$, we can take $w(1,1) = 1$, $w(2,1) = 3$, $w(1,2) = 2$, $w(2,2) = 4$.
\end{itemize}

In both cases we can check that $w\delta$ is locally admissible. The remaining cases are analogous.

\end{proof}

\begin{defn}
A set $W \subseteq \Z^d$ is called an \emph{$N$-shape} if it can be written as a union of translations of $[1,N]^d$, i.e. if there exists a set $S \subseteq \Z^d$ such that $W = \bigcup_{p \in S}\left(p+[1,N]^d\right)$. A set is called a \emph{co-$N$-shape} if it is the complement of an $N$-shape. Notice that every shape is a $1$-shape and co-$1$-shape.
\end{defn}

\begin{lem}
\label{lshape}
If a n.n. SFT $X$ satisfies $N$-fillability then, for any $N$-shape $W$ and every locally admissible configuration $\delta \in \Symb^{T}$, with $T \subseteq \Z^d \backslash W$, there exists $w \in \Symb^{W}$ such that $w\delta$ is locally admissible.

\begin{proof}
Let $W$ be an $N$-shape and $\delta \in \Symb^T$, for $T \subseteq \Z^d \backslash W$, a locally admissible configuration. Consider a minimal $S \subseteq \Z^d$ such that $W = \bigcup_{p \in S}\left(p+[1,N]^d\right)$, in the sense that $\bigcup_{p \in S'}\left(p+[1,N]^d\right) \subsetneq W$, for every $S' \subsetneq S$. Take an arbitrary $p^* \in S$. By $N$-fillability, consider $v \in \Symb^{\left(p^* + [1,N]^d\right)}$ such that $v\delta$ is locally admissible (notice that $T \subseteq \Z^d \backslash \left(p^* + [1,N]^d\right)$.

Now, take the set $W' = \bigcup_{p \in S \backslash\{p^*\}}{\left(p+[1,N]^d\right)}$. Notice that $W'$ is also an $N$-shape. By minimality of $S$, we have that $\emptyset \neq W \backslash W' \subseteq \left(p^*+[1,N]^d\right)$. Define $\delta' = v(W \backslash W')\delta$ and $T' = (W \backslash W') \cup T$. Then, $W'$ is an $N$-shape and $\delta' \in \Symb^{T'}$ is a locally admissible configuration, with $T' \subseteq \Z^d \backslash W'$ as in the beginning, but $W' \subsetneq W$.

Now, given $M \in \N$ and iterating the previous argument, we can always find $w \in \Symb^{W \cap \block_M}$ such that $w\delta$ is locally admissible. Since $M$ is arbitrary and $\Symb^{\Z^d}$ is a compact space, then there must exist $w \in \Symb^{W}$ such that $w\delta$ is locally admissible.
\end{proof}
\end{lem}

\begin{defn}
Given $N \in \N$, a shift space $X$ is said to be \emph{$N$-strongly irreducible with gap $g$} if for any pair of non-empty (disjoint) finite subsets $U,V \Subset \Z^d$ with separation $\dist(U,V) \geq g$ such that $U \cup V$ is a co-$N$-shape and,
\begin{equation}
\forall u \in \Symb^{U}, v \in \Symb^{V}: [u]_X,[v]_X \neq \emptyset \implies [uv]_X \neq \emptyset.
\end{equation}
\end{defn}

\begin{prop}
If a n.n. SFT $X$ satisfies $N$-fillability, then it is $N$-strongly irreducible with gap $g = 2$.
\end{prop}

\begin{proof}
Let $U,V \Subset \Z^d$ such that $\dist(U,V) \geq 2$ and $U \cup V$ is a co-$N$-shape, and $u \in \Symb^{U}$, $v \in \Symb^{V}$ such that $[u]_X,[v]_X \neq \emptyset$. Then, take $\delta = uv \in \Symb^{U \cup V}$ and $W = (U \cup V)^c$. Notice that $\delta$ is a locally admissible configuration ($u$ and $v$ are globally admissible and $\dist(U,V) \geq 2$), and $W$ is an $N$-shape. Then, by Lemma \ref{lshape}, there exists $w \in \Symb^{W}$ such that $x = w\delta$ is locally admissible. Then, $x$ is a locally admissible point (then, globally admissible) such that $x \in [uv]_X$.
\end{proof}

\begin{prop}
If a shift space $X$ is $N$-strongly irreducible with gap $g$, then $X$ is strongly irreducible with gap $g+2N$.
\end{prop}

\begin{proof}
Let $U,V \Subset \Z^d$ with $\dist(U,V) \geq g$, and $u \in \Symb^{U}$, $v \in \Symb^{V}$ such that $[u]_X,[v]_X \neq \emptyset$. Consider the partition $\Z^d = \coprod_{p \in N\Z^d}{\left(p + [1,N]^d\right)}$ and the sets:
\begin{align}
S_1	&	:= \left\{p \in N\Z^d: \left(p + [1,N]^d\right) \cap U \neq \emptyset \right\},	\\
S_2	&	:= \left\{p \in N\Z^d: \left(p + [1,N]^d\right) \cap V \neq \emptyset \right\}.
\end{align}

Notice that $U \subseteq U' := \coprod_{p \in S_1}{p + [1,N]^d}$ and $V \subseteq V' := \coprod_{p \in S_2}{p + [1,N]^d}$. Take $x \in [u]_X$ and $y \in [v]_X$, and consider the configurations $u' = x(U')$ and $v' = y(V')$. Then, we have that $[u']_X,[v']_X \neq \emptyset$, $U' \cup V'$ is a co-$N$-shape and $\dist(U',V') \geq \dist(U,V) - 2N \geq (g + 2N) - 2N = g$ so, by $N$-strong irreducibility, we conclude that $\emptyset \neq [u'v']_X \subseteq [uv]_X$.
\end{proof}

\begin{cor}
\label{fillstrirr}
If a n.n. SFT $X$ satisfies $N$-fillability, then it is strongly irreducible with gap $2(N+1)$.
\end{cor}

\begin{cor}
The n.n. SFT $\mathcal{C}_2(4)$ is strongly irreducible with gap $g = 6$.
\end{cor}

We have concluded $\mathcal{C}_2(k)$ is strongly irreducible if and only if $k \geq 4$ (the cases $k=2,3$ do not even satisfy the D-condition \cite{1-marcus}). On the other hand, $\mathcal{C}_2(k)$ satisfies TSSM (in particular, SSF) if and only if $k \geq 5$. In particular, TSSM fails when $k = 4$, as the next result shows.

\begin{prop}
\label{C4}
The n.n. SFT $\mathcal{C}_2(4)$ does not satisfy TSSM.
\end{prop}

\begin{proof}
Take $g \in \N$, consider the sets $U = \{(-2g,0)\}$, $V = \{(2g,0)\}$ and $S = [-2g,2g] \times \{-1,1\}$, and the configurations $u \in \Symb(4)^U$, $v \in \Symb(4)^V$ and $s \in \Symb(4)^S$ (see Figure \ref{checkerboard}) defined by $u = 3$, $v = 4$, and:
\begin{equation}
s((i,j)) = 
\begin{cases}
1 & \mbox{if ($j=1$ and $i \in 2\Z$) or ($j=-1$ and $i \notin 2\Z$)}, \\
2 & \mbox{if ($j=1$ and $i \notin 2\Z$) or ($j=-1$ and $i \in 2\Z$)}.
\end{cases}
\end{equation}

Then, it can be checked that $[us]_{\mathcal{C}_2(4)}, [sv]_{\mathcal{C}_2(4)} \neq \emptyset$. However, for all $x \in [us]_{\mathcal{C}_2(4)}$ and $y \in [sv]_{\mathcal{C}_2(4)}$ we have that $x((0,0)) = 3 \neq 4 = y((0,0))$. Therefore, $[usv]_{\mathcal{C}_2(4)} = \emptyset$. Since $g$ was arbitrary and $\dist(U,V) = 4g \geq g$, we conclude that $\mathcal{C}_2(4)$ does not satisfy TSSM.
\end{proof}

A by-product of the construction from the previous counterexample is the following result, which also illustrates how TSSM is related with SSM.

\begin{prop}
\label{C4SSM}
Let $\mu$ be a $\Z^2$ MRF such that $\supp(\mu) = \mathcal{C}_2(4)$. Then, $\mu$ cannot satisfy SSM.
\end{prop}

\begin{proof}
Let's suppose that there is a $\Z^2$ MRF $\mu$ with $\supp(\mu) = \mathcal{C}_2(4)$ that satisfies SSM with rate $f(n)$. Take $n_0 \in \N$ such that $f(n) < 1$, for all $n \geq n_0$. Consider the set $V = [-2n_0+1,2n_0-1] \times \{0\} \Subset \Z^2$ and its boundary $\partial V =  [-2n_0,2n_0] \times \{-1,1\} \cup \{(-2n_0,0)\} \cup  \{(0,2n_0)\} = U \cup S \cup V$, where $U$, $S$ and $V$ are as in Proposition \ref{C4}. Take $\delta_1,\delta_2 \in \Symb(4)^{\partial V}$ defined by $\delta_1(S) = \delta_2(S) = s$ (where $s$ is also as in Proposition \ref{C4}), $\delta_1((-2n_0,0)) = \delta_1((2n_0,0)) = 3$ and $\delta_2((-2n_0,0)) = \delta_2((2n_0,0)) = 4$. It is easy to see that $\delta_1$ and $\delta_2$ are both globally admissible and, in particular, $\mu(\delta_1),\mu(\delta_2) > 0$. Now, if we consider the configuration $w = 3$ with shape $W = \{(0,0)\}$, we have that:
\begin{equation}
1 = |1 - 0| = \left| \mu^{\delta_1}(w) - \mu^{\delta_2}(w) \right| \leq f(2n_0) < 1,
\end{equation}
which is a contradiction. Then, $\mu$ cannot satisfy SSM.
\end{proof}

\begin{rem}
It has been suggested \cite{1-salas} that the uniform Gibbs measure supported on $\mathcal{C}_2(4)$ satisfies exponential WSM. Here we have proven that SSM is not possible for any MRF supported on $\mathcal{C}_2(4)$ and for any rate, not necessarily exponential. The counterexample in Proposition \ref{C4SSM} corresponds to a family of very particular shapes where SSM fails and not what we could call a ``common shape'' (like $\block_n$, for example), but is enough for discarding the possibility of SSM if we stick to its definition. We also have to consider that this family of configurations (and other variations, with different colours and different narrow shapes) can appear as sub-configurations in more general shapes and still produce combinatorial long-range correlations.
\end{rem}

\begin{figure}[ht]
\centering
\includegraphics[scale = 0.75]{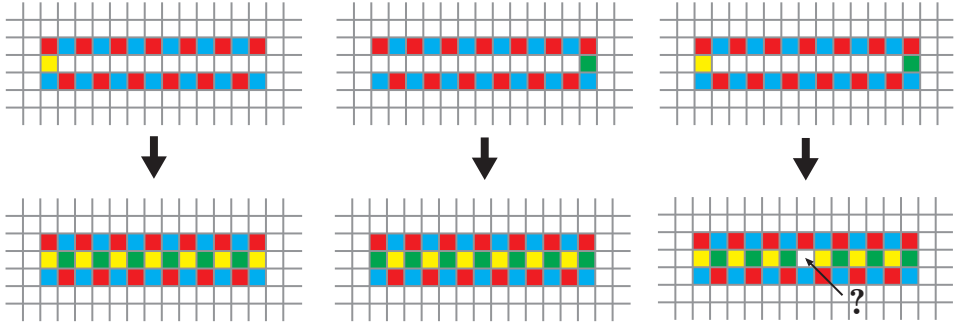}
\caption{Proof that $\mathcal{C}_2(4)$ does not satisfy TSSM nor SSM.}
\label{checkerboard}
\end{figure}

\subsection{A n.n. SFT that satisfies TSSM, but not SSF}
The Iceberg model was considered in \cite{1-burton} as an example of a strongly irreducible $\Z^2$ n.n. SFT with multiple measures of maximal entropy. Given $M \geq 2$, and the alphabet $\Symb(M) = \left\{-M,\dots,-1,+1,\dots,+M\right\}$, the Iceberg model $\mathcal{I}_M$ is defined as:
\begin{equation}
\mathcal{I}_M := \left\{x \in \Symb(M)^{\Z^d}: x(p) \cdot x(p+e_i) \geq -1, \mbox{ for all } p \in \Z^d, i = 1,\dots,d \right\}.
\end{equation}

In the following, we show that for every $M \geq 2$, the Iceberg model satisfies TSSM, but not SSF. In particular, this provides an example of a n.n. SFT satisfying TSSM with multiple measures of maximal entropy.

It is easy to see that $\mathcal{I}_M$ does not satisfy SSF, since $+M$ and $-M$ cannot be at distance less than $3$. In particular, we can take the configuration $\eta \in \Symb^{\partial\{0\}}$ given by $\eta(e_1) = \eta(e_2) = +M$ and $\eta(-e_1) = \eta(-e_2) = -M$, which does not remain locally admissible for any $a \in \Symb^{\{0\}}$. On the other hand, $\mathcal{I}_M$ satisfies TSSM, as the next proposition shows.

\begin{prop}
\label{iceberg}
For every $M \geq 2$, the Iceberg model $\mathcal{I}_M$ satisfies TSSM with gap $g = 3$.
\end{prop}

\begin{proof}
Consider Lemma \ref{lemSing} and take disjoint non-empty subsets $U,S,V \Subset \Z^2$ with $\dist(U,V) \geq 3$ and $|U| = |V| = 1$. Given $u \in \Symb^U$, $s \in \Symb^S$ and $v \in \Symb^V$, suppose that $[us]_{\mathcal{I}_M}, [sv]_{\mathcal{I}_M} \neq \emptyset$. Next, take $x \in [sv]_{\mathcal{I}_M}$ and define a new point $z$ given by:
\begin{equation}
z(p) = 
\begin{cases}
x(p)	& \mbox{ if } p \in S \cup V	,									\\
+1	&	\mbox{ if } p \in (S \cup V)^c \mbox{ and }	x(p) \in \{+1,\dots,+M\},	\\
-1	&	\mbox{ if } p \in (S \cup V)^c \mbox{ and } x(p) \in \{-M,\dots,-1\}.
\end{cases}
\end{equation}

It is not hard to see that $z$ is a valid point in $[sv]_{\mathcal{I}_M}$. Now, let's construct a point $y \in [usv]_{\mathcal{I}_M}$ from $z$.

\bigskip
\noindent
{\bf Case 1: $u = \pm 1$.} W.l.o.g., suppose that $u = +1$. Now, since $[us]_{\mathcal{I}_M} \neq \emptyset$, all the values in $z(\partial U \cap S)$ must belong to $\{-1,+1,\dots,+M\}$. On the other hand, since $\partial U \backslash S \subseteq (S \cup V)^c$, all the values in $z(\partial U \backslash S)$ belong to $\{-1,+1\}$. Then, all the values in $z(\partial U)$ belong to $\{-1,+1,\dots,+M\}$ and we can replace $z(U)$ by $+1$ in order to get a valid point $y$ from $z$, such that $y \in [usv]_{\mathcal{I}_M}$.

\bigskip
\noindent
{\bf Case 2: $u \neq \pm 1$.} W.l.o.g., suppose that $u = +M$. Then, all the values in $z(\partial U \cap S)$ belong to $\{+1,\dots,+M\}$. We claim that we can switch every $-1$ in $\partial U \backslash S$ to a $+1$. If it is not possible to do this for some site $p^* \in \partial U \backslash S$, then its neighbourhood $\partial \{p^*\}$ contains a site with value in $\{-M,\dots,-2\}$ and, in particular, different from $+1$ and $-1$. Then, $\partial \{p^*\}$ necessarily intersects $S$ (and not $V$, because $\dist(U,V) \geq 3$). Then, a site in $\partial \{p^*\} \cap S \neq \emptyset$ is fixed to some value in $\{-M,\dots,-2\}$ and then the site $p^*$ must take a value in $\{-M,\dots,-1\}$, given $s$. Therefore, $U$ cannot take a value in $\{+2,\dots,+M\}$, contradicting the fact that $[us]_{\mathcal{I}_M} \neq \emptyset$. Therefore, we can set all the values in $z(\partial U \backslash S)$ to $+1$. Let's call that point $z'$. Finally, if we replace $z'(U) = +1$ by $+M$, we obtain a valid point $y$ from $z'$ such that $y \in [usv]_{\mathcal{I}_M}$.

Then, we conclude that $\mathcal{I}_M$ satisfies TSSM with gap $g = 3$, for every $M \geq 2$.
\end{proof}

\begin{rem}
In particular, Proposition \ref{iceberg} provides an alternative way of checking the well-known fact that $\mathcal{I}_M$ is strongly irreducible.
\end{rem}

\subsection{Arbitrarily large gap, arbitrarily high rate}

Now we will present a variation of the Iceberg model. Notice that the Iceberg model can be regarded as a shift space where two ``disjoint" full shifts coexist (positives and negatives) separated by a boundary of $\pm1$s. In the following, we present a family of shift spaces that try to extend the idea of full shifts coexisting from the two in the Iceberg model to an arbitrary number. First, we will see that this variation gives a family of $\Z^d$ n.n. SFTs satisfying TSSM with gap $g$ but not $g-1$, for arbitrary $g \in \N$. Second, we will prove that any of these models admits the existence of n.n. Gibbs measures supported on them and satisfying exponential SSM with arbitrarily high rate, showing in particular (as far as we know, for the first time) that there are systems that satisfy SSM and TSSM, without satisfying any of the other stronger combinatorial mixing properties, like having a safe symbol or satisfying SSF.

Given $g,d \in \N$, consider the alphabet $\Symb_g = \left\{0,1,\dots,g\right\}$ and the n.n. SFT defined by:
\begin{equation}
X_g^d := \left\{x \in \Symb_g^{\Z^d}: \left|x(p) - x(p+e_i)\right| \leq 1, \mbox{ for all } p \in \Z^d, i=1,\dots,d\right\}.
\end{equation}

Notice that $X_0^d = \left\{0^{\Z^d}\right\}$ (a fixed point) and $X_1^d = \Symb_1^{\Z^d}$ (a full shift), so both satisfy TSSM with gap $g=0$ and $g=1$, respectively. Also, notice that $1$ is a safe symbol for $X_2^d$. 

\begin{prop}
\label{andes1}
The n.n. SFT $X_g^d$ satisfies TSSM with gap $g$ but not $g-1$.

\begin{proof}
First, let's see that $X_g^d$ does not satisfy TSSM with gap $g-1$. In fact, recall that TSSM with gap $g-1$ implies strong irreducibility with the same gap. However, if we consider two configurations on single sites with values $0$ and $g$, respectively, they cannot appear in the same point if they are separated by a distance less or equal to $g-1$, since the values in consecutive sites can only increase or decrease by at most $1$. Therefore, $X_g^d$ is not TSSM with gap $g-1$.

Now, let's prove that $X_g^d$ satisfies TSSM with gap $g$. Consider Lemma \ref{lemSing}, $p,q \in \Z^d$ with $\dist(p,q) \geq g$, and $S \Subset \Z^d \backslash \{p,q\}$. Given $u \in \Symb_g^{\{p\}}$, $v \in \Symb_g^{\{q\}}$ and $s \in \Symb_g^{S}$, suppose that $[us]_{X_g^d},[sv]_{X_g^d} \neq \emptyset$. We want to prove that $[usv]_{X_g^d} \neq \emptyset$.

Since $[sv]_{X_g^d} \neq \emptyset$, we can consider a point $x \in [sv]_{X_g^d}$. If $x(p) = u$, we are done. W.l.o.g., suppose that $x(p) < u$ (the case $x(p) > u$ is analogous). We proceed by finding a valid point $x'$ such that $x'(S) = s$, $x'(q) = v$ and $x'(p) = x(p) + 1$. Iterating this process $u - x(p)$ times, we conclude. For doing this, notice that the only obstruction for increasing by $1$ the point $x$ at $p$ are the values of neighbours of $p$ strictly below $x(p)$. Considering this fact, we introduce a \emph{(directed) graph of descending paths} $\mathcal{D}(x,p) = (\mathcal{V}_g(x,p),\mathcal{E}_g(x,p))$, where $\mathcal{V}_0(x,p) =  \{p\}$, $\mathcal{E}_0(x,p) = \emptyset$ and, for $n \geq 1$:
\begin{align}
\mathcal{V}_{n+1}(x,p)	& = \mathcal{V}_n(x,p) \cup \bigcup_{\substack{r:~\dist(r, V_n(x,p)) = 1\\x(r) = x(p) - n}} \{r\},	\\
\mathcal{E}_{n+1}(x,p)	& = \mathcal{E}_n(x,p) \cup \bigcup_{\substack{r:~\dist(r, V_n(x,p)) = 1\\x(r) = x(p) - n}} \left\{(s,r):  s \in V_n(x,p), s \sim r\right\}.
\end{align}

Notice that, since $x(p) < g$, the recurrence stabilizes for some $n < g$, i.e. $\mathcal{V}_{n}(x,p) = \mathcal{V}_{g-1}(x,p)$ and $\mathcal{V}_{n}(x,p) = \mathcal{V}_{g-1}(x,p)$, for every $n \geq g$. In particular, the vertices that $\mathcal{D}(x,p)$ reaches are sites at distance at most $g-1$ from $p$, and the site $q$ cannot belong to the graph. Now, suppose that a site from $S$ belongs to $\mathcal{D}(x,p)$. If that is the case, the value at $p$ of any point in $[s]_{X_g^d}$ would be forced to be at most $x(p)$ (since the graph is strictly decreasing from $p$ to $S$), which contradicts the fact that $[us]_{X_g^d} \neq \emptyset$. 

Then, neither $q$ nor any element of $S$ belongs to $\mathcal{D}(x,p)$, so if we modify the values of $\mathcal{D}(x,p)$ in a valid way, we will still obtain a valid point $x'$ such that $x'(S) = s$ and $x'(q) = v$. Now, take the set $D =  \mathcal{V}(\mathcal{D}(x,p))$ and consider the point $x'$ such that:
\begin{align}
x'(D) = x(D) + 1,	&	\mbox{ and }	x'(\Z^d \backslash D) = x(\Z^d \backslash D).
\end{align}
where $x(D) + 1$ represents the configuration obtained from $x(D)$ after adding $1$ in every site. We claim that $x'$ is a valid point. To see this, we only need to check that the difference between values of vertices in an arbitrary edge is at most $1$. If both ends are in $D$ or in $\Z^d \backslash D$, it is clear that the edge is valid since the original point $x$ was a valid point, and adding $1$ to both ends does not affect the difference. If one end is in $r_1 \in D$ and the other one is in $r_2 \in \Z^d \backslash D$, then $x(r_1) \leq x(r_2)$, necessarily (if not, $x(r_1) > x(r_2)$, and $r_2$ would be part of the graph of descending paths). Since $\left|x(r_1) - x(r_2)\right| \leq 1$ and $x(r_1) \leq x(r_2)$, then $x(r_2) - x(r_1) \in \{0,1\}$. Therefore, $x'(r_1) - x'(r_2) = (x(r_1)+1) - x(r_2) = 1 - (x(r_2) - x(r_1)) \in \{1,0\}$, so $|x'(r_1) - x'(r_2)| \leq 1$. Then, we conclude that $x' \in X_g^d$ and $x'(S) = s$, $x'(q) = v$ and $x'(p) = x(p) + 1$, as we wanted.
\end{proof}
\end{prop}

\begin{prop}
\label{andes2}
For any $g,d \in \N$, there exists a n.n. Gibbs measure on $X_g^d$ satisfying exponential SSM with rate $f(n) = Ce^{-\alpha n}$, for some $C,\alpha > 0$, where $\alpha$ can be chosen to be arbitrarily large.
\end{prop}

Before proving Proposition \ref{andes2}, we will provide some auxiliary results. From now on, fix $g,d \in \N$ and a shape $W \Subset \Z^d$. We consider the partial order $\preccurlyeq$ on $\Symb_g^W$ obtained by extending coordinate-wise the natural total order on $\Symb_g$ to $W$, i.e. $w \preccurlyeq w'$ if and only if $w(p) \leq w'(p)$, for all $p \in W$.

\begin{lem}
\label{lem1}
Given $\delta \in \Leng_{\partial W}(X_g^d)$, there is a unique configuration $\theta_\delta \in \Leng_{W}(X_g^d)$ such that $\theta_\delta \delta$ is globally admissible and $w \preccurlyeq \theta_\delta$, for any other configuration $w \in \Leng_{W}(X_g^d)$ such that $w\delta$ is locally admissible. We call $\theta_\delta$ the \emph{maximal configuration for $\delta$}.
\end{lem}

\begin{proof}
Given $\delta \in \Leng_{\partial W}(X_g^d)$, suppose that there exist two incomparable configurations $\theta_1, \theta_2 \in \Leng_{W}(X_g^d)$ such that $w \preccurlyeq \theta_j$ ($j =1,2$), for every $w \in \Leng_{W}(X_g^d)$ comparable with $\theta_j$ and such that $w\delta$ is locally admissible. Consider the configuration $\theta^* \in \Symb_g^W$ obtained by taking the site-wise maximum of $\theta_1$ and $\theta_2$. In other words, $\theta^*(p) := \max(\theta_1(p),\theta_2(p))$, for every $p \in W$. We claim that $\theta^*\delta \in \Leng_{W}(X_g^d)$. W.l.o.g., we can assume that there is a partition $W = W_1 \sqcup W_2$ such that $\theta^*(p) = \theta_j(p)$, for every $p \in W_j$ ($j = 1,2$). Take an arbitrary $p \in W \cup \partial W$ and $i \in \{1,\dots,d\}$. If $\{p,p+e_i\} \subseteq W_j$ for some $j$, then $|\theta^*(p)-\theta^*(p+e_i)| = |\theta_j(p)-\theta_j(p+e_i)| \leq 1$. If $p \in W_1$ and $p+e_i \in W_2$, then:
\begin{equation}
-1 \leq \theta_2(p)-\theta_2(p+e_i) \leq \theta^*(p)-\theta^*(p+e_i) \leq \theta_1(p)-\theta_1(p+e_i) \leq 1,
\end{equation}
so $|\theta^*(p)-\theta^*(p+e_i)| \leq 1$. If $p \in W_2$ and $p+e_i \in W_1$, the proof is analogous. Finally, if $p$ or $p+e_i$ is in $\partial W$, then we also have $|\theta^*\delta(p)-\theta^*\delta(p+e_i)| \leq 1$, because $\theta_1\delta$ and $\theta_2\delta$ are locally admissible. Then, $\theta^*\delta$ is locally admissible (and therefore, since $X_g^d$ is a n.n SFT,  globally admissible), $\theta_j \preccurlyeq \theta^*$ ($j = 1, 2$) and $\theta_j \neq \theta^*$, contradicting the maximality of $\theta_1$ and $\theta_2$. Therefore, since $\Symb_g^W$ is finite, there must exist one and only one maximal configuration $\theta_\delta$.
\end{proof}

\begin{lem}
\label{lem2}
Given $\delta_1,\delta_2 \in \Leng_{\partial W}(X_g^d)$, we have that:
\begin{equation}
\Sigma_W(\theta_{\delta_1},\theta_{\delta_2}) \subseteq \neig_g(\Sigma_{\partial W}(\delta_1,\delta_2)) \cap W.
\end{equation}
\end{lem}

\begin{proof}
 Consider the maximal configurations $\theta_{\delta_j}$ ($j = 1,2$) and suppose $\theta_{\delta_1}(p) \neq \theta_{\delta_2}(p)$, for some $p \in W$ such that $\dist(p,\Sigma_{\partial W}(\delta_1,\delta_2)) \geq g$. W.l.o.g., suppose that $\theta_{\delta_1}(p) > \theta_{\delta_2}(p)$. Considering $u := \theta_{\delta_1}(p)$, $s := \delta_1(\partial W \backslash \Sigma_{\partial W}(\delta_1,\delta_2)) = \delta_2(\partial W \backslash \Sigma_{\partial W}(\delta_1,\delta_2))$ and $v := \delta_2(\Sigma_{\partial W}(\delta_1,\delta_2))$, we have that $[us]_{X_g^d}, [sv]_{X_g^d} \neq \emptyset$, so $\emptyset \neq [usv]_{X_g^d} =  [\theta_{\delta_1}(p)\delta_2]_{X_g^d}$, due to the TSSM property. Take any $x \in [\theta_{\delta_1}(p)\delta_2]_{X_g^d}$ and consider $w = x(W)$. Then, $w \preccurlyeq \theta_{\delta_2}$, but $w(p) > \theta_{\delta_2}(p)$, which is a contradiction. Therefore, $\Sigma_W(\theta_{\delta_1},\theta_{\delta_2}) \subseteq \neig_g(\Sigma_{\partial W}(\delta_1,\delta_2)) \cap W$.
\end{proof}

We will use the following result.

\begin{thm}[{\cite[Theorem 1]{1-berg}}]
\label{vandenberg}
Let $\mu$ be an MRF. For every $W \Subset \Z^d$ and each pair $\delta_1,\delta_2 \in \Symb^{\partial W}$, there exists a coupling $((w_1(p),w_2(p)), p \in W)$ of $\mu^{\delta_1}$ and $\mu^{\delta_2}$ (whose distribution we denote by $\mathbb{P}$), such that for each $p \in W$, $w_1(p) \neq w_2(p)$ if and only if there is a path of disagreement (i.e. a path $\camino$ such that $w_1(q) \neq w_2(q)$, for all $q \in \camino$) from $p$ to $\Sigma_{\partial W}(\delta_1,\delta_2)$ ($\mathbb{P}$-a.s.).
\end{thm}

Consider a parameter $\lambda > 0$ to be determined. Given configurations $m \in \Symb_g^{\{0\}}$ and $mn \in \Symb_g^{\{0,e_i\}}$, for an arbitrary $i = 1,\dots,d$, we define a n.n. interaction $\Phi$ on $\Z^d$ given by:
\begin{align}
\Phi(m)  = -m\log\lambda,	&	\mbox{ and }	\Phi(mn) = \begin{cases} 0 & \mbox{if } |m - n| \leq 1, \\ \infty & \mbox{otherwise.}\end{cases}
\end{align}

Clearly, $\mathsf{X}(\Phi) = X_g^d$. Now, fix $\mu$ any n.n. Gibbs measure for $\Phi$. Since $X_g^d$ satisfies the D-condition, by Proposition \ref{dconsupp} we have that $\supp(\mu) = X_g^d$.

\begin{lem}
\label{lem3}
Given $\delta \in \Leng_{\partial W}(X_g^d)$, a subset $U \subseteq W$ and $k \leq |U|$,
\begin{equation}
\mu^\delta\left(\left|\left\{ q \in U: w(q) < \theta_\delta(q)\right\}\right| \geq k\right) \leq (g+1)^{|U||\neig_g(0)|} \lambda^{-k}.
\end{equation}
\end{lem}

\begin{proof}
Consider an arbitrary configuration $w \in \Symb_g^{W}$ such that $w\delta$ is locally admissible. Notice that the boundary $\partial T$ of the set $T := W \cap \neig_g(U)$ can be decomposed into two subsets, namely $V := W \cap \partial \neig_g(U)$ and $S := \partial W \cap \partial \neig_g(U)$. Then, we can consider the boundary configuration $\eta := w(V)\delta(S) \in \Symb_g^{\partial T}$ and the corresponding maximal configurations $\theta_\delta \in \Symb_g^{W}$ and $\theta_\eta \in \Symb_g^{T}$, given by Lemma \ref{lem1}.

\begin{figure}[ht]
\centering
\includegraphics[scale = 0.7]{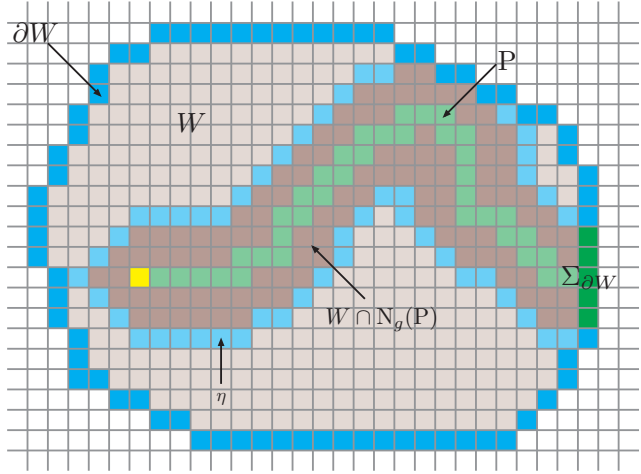}
\caption{Decomposition in the proof of Lemma \ref{lem3}, used later in the proof of Proposition \ref{andes2}.}
\label{andes}
\end{figure}

Notice that $\theta_\delta(U)\delta(S)$ and $\delta(S)w(V)$ are globally admissible, and $\dist(U,V) \geq g$. Then, by TSSM, $\theta_\delta(U)\delta(S)w(V) = \theta_\delta(U)\eta$ is globally admissible, too. By maximality of $\theta_\eta$, we have that $\theta_\eta(q) \geq \theta_\delta(q)$, for all $q \in U$. Similarly, since $w(W \backslash T)\theta_\eta\delta$ is locally admissible, we have that $\theta_\eta(q) \leq \theta_\delta(q)$, for all $q \in U$. Therefore, $\theta_\eta(U) = \theta_\delta(U)$.

Now, suppose that $w$ is such that $\left|\left\{ q \in U: w(q) < \theta_\delta(q)\right\}\right| \geq k$, for some $k \leq |U|$. Then, by the MRF property:
\begin{equation}
\frac{\mu^\delta\left(\theta_\eta \middle\vert w(W \backslash T))\right)}{\mu^\delta\left(w(T) \middle\vert w(W \backslash T)\right)} = \frac{\mu^\eta\left(\theta_\eta\right)}{\mu^\eta\left(w(T)\right)} \geq \lambda^{k}.
\end{equation}

Therefore,
\begin{equation}
\mu^\delta\left(w(T) \middle\vert w(W \backslash T)\right) \leq \lambda^{-k}.
\end{equation}

Next, by integrating over all configurations $v \in \Leng_{W \backslash T}(X_g^d)$ such that $vw(T)\delta$ is locally admissible, we have that:
\begin{equation}
\mu^\delta(w(T)) \leq \lambda^{-k}.
\end{equation}

Notice that $|T| \leq \left|\neig_g(U)\right| \leq |U||\neig_g(0)|$. In particular, $\left|\Leng_{T}\left(X_g^d\right)\right| \leq (g+1)^{|U||\neig_g(0)|}$. Then, since $w$ was arbitrary:
\begin{align}
\mu^\delta\left(\left|\left\{ q \in U: w(q) < \theta_\delta(q)\right\}\right| \geq k\right)		&	\leq \sum_{\substack{u \in \Leng_T(X_g^d):\\\left|\left\{ q \in U: u(q) < \theta_\delta(q)\right\}\right| \geq k}}\mu^\delta(u)	\\
																&	\leq (g+1)^{|U||\neig_g(0)|} \lambda^{-k}.
\end{align}
\end{proof}

Now we are in good shape for finishing the proof of Proposition \ref{andes2}.

\begin{proof}[Proof (of Proposition \ref{andes2})]
Take $p \in W$, $u \in \Symb_g^{\{p\}}$ and $\delta_1,\delta_2 \in \Leng_{\partial W}(X_g^d)$. W.l.o.g., suppose that $\dist\left(p,\Sigma_{\partial W}(\delta_1,\delta_2)\right) = n > g$. By Theorem \ref{vandenberg}, we have that:
\begin{align}
\left|\mu^{\delta_1}(u) - \mu^{\delta_2}(v)\right|	&	=	\left|\mathbb{P}\left(w_1(p) = u\right) - \mathbb{P}\left(w_2(p) = u\right)\right|	\\
									&	\leq	\mathbb{P}(w_1(p) \neq w_2(p))	\\
									&	=	\mathbb{P}\left(\exists \mbox{ path of disagr. from } p \mbox{ to } \Sigma_{\partial W}(\delta_1,\delta_2) \right)	\\
									&	\leq	\mathbb{P}\left(\exists \mbox{ path of disagr. from } p \mbox{ to } \neig_g(\Sigma_{\partial W}(\delta_1,\delta_2)) \right).
\end{align}

When considering a path of disagreement $\camino$ from $p$ to $\neig_g(\Sigma_{\partial W}(\delta_1,\delta_2))$, we can assume that $\camino \subseteq W \backslash \neig_g(\Sigma_{\partial W}(\delta_1,\delta_2))$ and $|\camino| \geq n-g$. By Lemma \ref{lem2}, we have that $\theta_{\delta_1}(\camino) = \theta_{\delta_2}(\camino) =: \theta \in \Leng_{\camino}(X_g^d)$. Since $\camino$ is a path of disagreement, $w_1(q) < w_2(q) \leq \theta(q)$ or $w_2(q) < w_1(q) \leq \theta(q)$, for every $q \in \camino$. In consequence, and using Lemma \ref{lem3},
\begin{align}
	&	\mathbb{P}\left(\exists \mbox{ path } \camino \mbox{ of disagr. from } p \mbox{ to } \neig_g(\Sigma_{\partial W}(\delta_1,\delta_2)) \right) \\
\leq	&~	\sum_{k = n-g}^\infty \sum_{|\camino| = k} \mathbb{P}\left(\camino \mbox{ is a path of disagr. from } p \mbox{ to } \neig_g(\Sigma_{\partial W}(\delta_1,\delta_2)) \right)	\\
\leq	&~	\sum_{k = n-g}^\infty \sum_{|\camino| = k} \mu^{\delta_1}\left(\left|\left\{ q \in \camino: w_1(q) < \theta(q)\right\}\right| \geq \frac{k}{2}\right) \\
	&~	+ \sum_{k = n-g}^\infty \sum_{|\camino| = k} \mu^{\delta_2}\left(\left|\left\{ q \in \camino: w_2(q) < \theta(q)\right\}\right| \geq \frac{k}{2}\right)	\nonumber\\
\leq	&~	2\sum_{k = n-g}^\infty \sum_{|\camino| = k} (g+1)^{k|\neig_g(0)|}\lambda^{-\frac{k}{2}}	\\
\leq	&~	2\sum_{k = n-g}^\infty 2d(2d-1)^k \left(\frac{(g+1)^{|\neig_g(0)|}}{\lambda^{1/2}}\right)^k 	\\	
=	&~	4d\sum_{k = n-g}^\infty \left(\frac{(2d-1)(g+1)^{|\neig_g(0)|}}{\lambda^{1/2}}\right)^k =	 Ce^{-\alpha n} ,
\end{align}
if $\beta := \frac{(2d-1)(g+1)^{|\neig_g(0)|}}{\lambda^{1/2}} < 1$, $\alpha := -\log \beta = \frac{1}{2}\log\lambda - \log(2d-1) - |\neig_g(0)|\log(g+1)$ and $C := \left(\frac{4d}{\beta^g(1-\beta)}\right)$. Notice that $|\neig_g(0)| \leq (2g+1)^d$. Then, it suffices to take:
\begin{equation}
\lambda	>	(2d-1)^2(g+1)^{(2g+1)^{2d}}.
\end{equation}

Finally, by Lemma \ref{ssmSing}, we conclude the (exponential) SSM property.
\end{proof}

The preceding proof is based on the modification of an approach used in \cite{1-brightwell} for proving uniqueness of Gibbs measures with constraints defined in terms of \emph{dismantlable graphs}. Here we use the coupling from Theorem \ref{vandenberg} (see \cite[Theorem 1]{1-berg}), which is different from the coupling used in \cite{1-brightwell} (see \cite[Theorem 1]{2-berg}). It is very likely that the bounds can be improved (using self avoiding paths, etc.). W.l.o.g., we could have also assumed that $\Sigma_{\partial W}(\delta_1,\delta_2) \subseteq \neig_g(p)$, for some $p \in \partial W$, thanks to Corollary \ref{boundary}. Also, notice that Proposition \ref{andes2} gives us an alternative way to prove TSSM for $X_g^2$, since the rate of decay can be arbitrarily large by adjusting $\lambda$ (in particular, larger than $4\log(g+1)$) and Theorem \ref{rate} applies.

\begin{note}
Proposition \ref{andes2} can be easily adapted to the hard-core model case (notice that the hard-core model is like $X_2^d$ but with $11$ forbidden, and this is not a problem for using the same arguments of the proof given here).
\end{note}

\section{Pressure representation}
\label{PressureRep}

When dealing with pressure representation, it is useful to consider an order in the lattice. A natural one is the so-called \emph{lexicographic order} on $\Z^d$, where $q \prec p$ if and only if $q \neq p$ and, for the smallest $i$ for which $q_i \neq p_i$, $q_i$ is strictly smaller than $p_i$. Considering $\prec$, we define the \emph{lexicographic past $\past$} of $\Z^d$ as the set $\past := \left\{p \in \Z^d: p \prec 0\right\}$ . Given $n \in \N$, we also define the set $\past_n := \past \cap \block_n$.

Given a shift-invariant measure $\mu$ on $\Symb^{\Z^d}$, we define $p_{\mu,n}(x) := p_{\mu,\past_n}(x)$ (recall Definition \ref{Pmu}). By martingale convergence, we can define $p_\mu(x) := \lim_{n \rightarrow \infty} p_{\mu,n}(x)$, that exists for $\mu$-a.e. $x \in \supp(\mu)$. Then, the \emph{information function $I_\mu$} is $\mu$-a.e. defined as:
\begin{equation}
I_\mu(x) := -\log p_\mu(x).
\end{equation}

It is known (see \cite[Theorem 15.12]{1-georgii} or \cite[Theorem 2.4]{1-krengel}) that the measure-theoretic entropy of $\mu$ can be expressed as:
\begin{equation}
\label{entropyPast}
h(\mu) = \int{I_\mu}d\mu.
\end{equation}

When applied to an equilibrium state $\mu$ for a function $f$, Equation (\ref{entropyPast}) clearly implies that:
\begin{equation}
P(f) = \int{\left(I_\mu + f\right)}d\mu.
\end{equation}

For certain classes of equilibrium states and Gibbs measures, sometimes there are even simpler representations for the pressure. A recent example of this was given by D. Gamarnik and D. Katz in \cite[Theorem 1]{1-gamarnik}, who showed that for any n.n. Gibbs measure $\mu$ for a n.n. interaction $\Phi$ which has the SSM property and such that $\mathsf{X}(\Phi)$ contains a safe symbol $0$:
\begin{equation}
P(\Phi) = I_\mu\left(0^{\Z^d}\right) + A_\Phi\left(0^{\Z^d}\right).
\end{equation}

Here, $0^{\Z^d} \in \Symb^{\Z^d}$ is the configuration on $\Z^d$ which is $0$ at every site of $\Z^d$. Notice that:
\begin{equation}
I_\mu\left(0^{\Z^d}\right) + A_\Phi\left(0^{\Z^d}\right) = \int{\left(I_\mu + A_\Phi\right)}d\nu_0,
\end{equation}
where $\nu_0$ is the measure supported on the fixed point $0^{\Z^d}$. They used this simple representation to give a polynomial time approximation algorithm for $P(\Phi)$ in certain cases (the hard-core model, in particular). Later, B. Marcus and R. Pavlov \cite{1-marcus} weakened the hypothesis and extended their results for pressure representation, obtaining the following corollary.

\begin{cor}[\cite{1-marcus}]
\label{corRep}
Let $\Phi$ be a n.n. interaction, $\mu$ a Gibbs measure for $\Phi$, and $\nu$ a shift-invariant measure with $\supp(\nu) \subseteq \mathsf{X}(\Phi)$ such that:
\begin{itemize}
\item $\mathsf{X}(\Phi)$ satisfies SSF, and
\item $\mu$ satisfies SSM.
\end{itemize}

Then, $P(\Phi) = \int{\left(I_\mu + A_\Phi\right)}d\nu$.
\end{cor}

Corollary \ref{corRep} relied on a more technical theorem from \cite[Theorem 3.1]{1-marcus}. Here we extend that result from fully supported Gibbs measures to (not necessarily fully supported) equilibrium states, something also necessary for our purposes (for example, see Corollary \ref{HighRep}). First, a couple of definitions.

We define $\lim_{S \rightarrow \past} p_{\mu,S}(x)$ to mean that there exists $L \in \R$ such that for any $\epsilon > 0$, there is $n \in \N$ such that for all $\past_n \subseteq S \Subset \past$, $|p_{\mu,S}(x) - L| < \epsilon$. Given $x$, by definition $L = p_{\mu}(x)$, if such $L$ exists. In addition, for shift-invariant measures $\mu$ and $\nu$ on $\Symb^{\Z^d}$, with $\supp(\nu) \subseteq \supp(\mu)$, we define:
\begin{equation}
\Cmu^{-}(\nu) := \inf\left\{p_{\mu,S}(x): x \in \supp(\nu), S \Subset \past \right\}.
\end{equation}

Notice that $\Cmu^{-}(\nu) \geq \Cmu$. We have the following theorem.

\begin{thm}
\label{thmA}
Let $\Phi$ be a n.n. interaction, $\mu$ an equilibrium state for $\Phi$, and $\nu$ a shift-invariant measure with $\supp(\nu) \subseteq \supp(\mu)$ such that:
\begin{itemize}
\item[(A1)] $\supp(\mu)$ satisfies the D-condition,
\item[(A2)] $\lim_{S \rightarrow \past} p_{\mu,S}(x) = p_{\mu}(x)$ uniformly over $x \in \supp(\nu)$, and
\item[(A3)] $\Cmu^{-}(\nu) > 0$.
\end{itemize}

Then, $P(\Phi) = \int{\left(I_\mu + A_\Phi\right)}d\nu$.
\end{thm}

\begin{proof}
We follow the proof of \cite[Theorem 3.1]{1-marcus} very closely. Let's denote $X = \mathsf{X}(\Phi)$ and $Y = \supp(\mu)$. For any $S \Subset \Z^d$, $w \in \Leng_S(Y)$ if and only if $\mu(w) > 0$. Choose $\ell < 0$ and $L > 0$ to be lower and upper bounds on finite values of $\Phi$,  respectively. Let $S_n$ and $T_n$ be as in the definition of the D-condition. Fix $n \in \N$ and let $R_n = |T_n| - |S_n|$. Note that for any $w \in \Leng_{S_n}(Y)$,
\begin{equation}
\label{eq7-1}
\mu(w) = \sum_{\delta \in \Leng_{\partial T_n}(Y)} \mu(w | \delta)\mu(\delta).
\end{equation}

For any such $w$ and $\delta$, in a very similar way to \cite[Theorem 3.1]{1-marcus} but considering the D-condition on $Y$ rather than on $X$, we have:
\begin{equation}
\label{eq7-2}
\frac{e^{-U_{S_n}^{\Phi}(w)}}{Z^{\Phi}(S_n)} e^{-C_dR_n(\log|\Symb| + L - \ell)}	\leq \mu(w | \delta) \leq \frac{e^{-U_{S_n}^{\Phi}(w)}}{Z^{\Phi}(S_n,Y)} e^{C_dR_n(\log|\Symb| + L - \ell)},
\end{equation}
for some constant $C_d \geq 1$ and $Z^{\Phi}(S_n,Y) := \sum_{v \in \Leng_{S_n}(Y)} e^{-U_{S_n}^{\Phi}(v)}$. Then, since $\sum_{\delta \in \Leng_{\partial T_n}(Y)}\mu(\delta) = 1$, we can combine Equation \ref{eq7-1} and Equation \ref{eq7-2} to see that:
\begin{equation}
\label{eq7-3}
\gamma^{-R_n} \leq \mu(w)Z^{\Phi}(S_n)e^{U_{S_n}^{\Phi}(w)} \leq \gamma^{R_n} \frac{Z^{\Phi}(S_n)}{Z^{\Phi}(S_n,Y)},
\end{equation}
where $\gamma := e^{C_d(\log|\Symb| + L - \ell)} > 0$ and $w = x(S_n)$, for a given $x \in Y$. Therefore, since $\mu$ is an equilibrium state and $\log$ is a concave function, by Jensen's inequality:
\begin{align}
P(\Phi)	&	=			\int{\left(I_\mu + A_\Phi\right)}d\mu	\\
		&	=			\lim_{n \to \infty}\frac{1}{|S_n|}\sum_{w \in \Leng_{S_n}(Y)}\left(-\log(\mu(w))\mu(w) - U_{S_n}^{\Phi}(w)\right)	\\
		&	=			\lim_{n \to \infty}\frac{1}{|S_n|}\sum_{w \in \Leng_{S_n}(Y)}\log\left(\frac{e^{-U_{S_n}^{\Phi}(w)}}{\mu(w)}\right)\mu(w)	\\
		&	\leq	\lim_{n \to \infty}\frac{1}{|S_n|}\log\left(\sum_{w \in \Leng_{S_n}(Y)} e^{-U_{S_n}^{\Phi}(w)}\right)	\\	
		&	=			\lim_{n \to \infty} \frac{\log Z^{\Phi}(S_n,Y)}{|S_n|}	\\
		&	\leq	\lim_{n \to \infty} \frac{\log Z^{\Phi}(S_n)}{|S_n|}	 = P(\Phi),
\end{align}
where we have used $Z^{\Phi}(S_n,Y) \leq Z^{\Phi}(S_n)$. Then, taking logarithms and dividing by $|S_n|$ in Equation \ref{eq7-3}:
\begin{align}
- \frac{R_n}{|S_n|} \log\gamma	&	\leq \frac{1}{|S_n|} \left(\log\mu(w) + \log Z^{\Phi}(S_n) + U_{S_n}^{\Phi}(w)\right)	\\
						&	\leq \frac{R_n}{|S_n|} \log\gamma  + \frac{\log Z^{\Phi}(S_n)}{|S_n|} - \frac{\log Z^{\Phi}(S_n,Y)}{|S_n|},
\end{align}
and, given that $\frac{R_n}{|S_n|} \to 0$ and $\lim_{n \to \infty} \frac{\log Z^{\Phi}(S_n,Y)}{|S_n|} = P(\Phi)$, this implies:
\begin{equation}
\label{conv}
\frac{1}{|S_n|}\left(\log Z^{\Phi}(S_n) + \log\mu(x(S_n)) + U_{S_n}^{\Phi}(x(S_n))\right) \to 0,
\end{equation}
uniformly in $x \in \supp(\mu)$, since $\gamma$, $R_n$, $|S_n|$, $Z^{\Phi}(S_n)$ and $Z^{\Phi}(S_n,Y)$ do not depend on $x$. Having this, the proof follows exactly as in \cite[Theorem 3.1]{1-marcus}.
\end{proof}

Considering the TSSM property, we have the following result.

\begin{cor}
\label{TSSMRep}
Let $\Phi$ be a n.n. interaction, $\mu$ an equilibrium state for $\Phi$, and $\nu$ a shift-invariant measure with $\supp(\nu) \subseteq \supp(\mu)$ such that:
\begin{itemize}
\item $\supp(\mu)$ satisfies TSSM, and
\item $\mu$ satisfies SSM.
\end{itemize}

Then, $P(\Phi) = \int{\left(I_\mu + A_\Phi\right)}d\nu$.
\end{cor}

\begin{proof}
This follows from Theorem \ref{thmA}: (A1) is implied by TSSM, since TSSM implies the D-condition; (A2) is implied by SSM (see \cite[Proposition 2.14]{1-marcus}); and (A3) is implied by TSSM (see Proposition \ref{TSSM-Cmu}), considering that $\Cmu^-(\nu) \geq \Cmu$.
\end{proof}

\begin{cor}
\label{HighRep}
Let $\mu$ be a $\Z^2$ MRF that satisfies exponential SSM with rate $\alpha > 4\log|\Symb|$. If $\mu$ is an equilibrium state for a n.n. interaction $\Phi$, we have that $P(\Phi) = \int{\left(I_\mu + A_\Phi\right)}d\nu$, for every shift-invariant measure $\nu$ such that $\supp(\nu) \subseteq \supp(\mu)$.
\end{cor}

Notice that, in contrast to preceding results, no mixing condition on the support is explicitly needed in Corollary \ref{HighRep}.

\section{Algorithmic implications}
\label{PressureApp}

In this section we give algorithmic results related with TSSM and pressure approximation. For the latter, we make heavy use of the representation results from the previous section.

\begin{prop}[\cite{1-gamarnik}]
Let $\Phi$ be the n.n. interaction corresponding to the hard-core model on $\Z^d$ with activity $\lambda > 0$. If:
\begin{equation}
\lambda < \lambda_c(\mathbb{T}_{2d}) := \frac{(2d-1)^{2d-1}}{(2d-2)^{2d}},
\end{equation}
then there is an algorithm to compute $P(\Phi)$ to within $\epsilon$ in time $\mathrm{poly}\left(\frac{1}{\epsilon}\right)$.
\end{prop}

\begin{note}
The value $\lambda_c(\mathbb{T}_{2d})$ corresponds to the critical activity of the hard-core model in the $2d$-regular tree $\mathbb{T}_{2d}$. This model satisfies exponential SSM (an extension of Definition \ref{SSMspec} to arbitrary graphs) if $\lambda < \lambda_c(\mathbb{T}_{2d})$. It is also known that the partition function of the hard-core model with $\lambda < \lambda_c(\mathbb{T}_{2d})$ in any finite graph of degree $2d$ can be efficiently approximated (for these and more results, see the fundamental work of D. Weitz in \cite{1-weitz}).
\end{note}

\begin{prop}[\cite{1-marcus}]
\label{algSSF}
Let $\Phi$ be a n.n. interaction and $\mu$ a Gibbs measure for $\Phi$ such that:
\begin{itemize}
\item $\mathsf{X}(\Phi)$ satisfies SSF, and
\item $\mu$ satisfies exponential SSM.
\end{itemize}

Then, there is an algorithm to compute $P(\Phi)$ to within $\epsilon$ in time $e^{O\left((\log(1/\epsilon))^{d-1}\right)}$.
\end{prop}

Note that in the case $d = 2$, Proposition \ref{algSSF} gives a polynomial time approximation algorithm. In Proposition \ref{algorithmProp}, we will extend this result by relaxing the mixing property in the support. First, some extra results.

\begin{lem}
\label{lemPer}
Let $X$ be a non-empty $\Z^d$ strongly irreducible shift space with gap $g$. Then, for all $S \Subset \Z^d$, $u \in \Symb^S$ and $x \in X$, $u$ is globally admissible if and only if there exists $y \in X$ such that:
\begin{eqnarray}
y(S) = u,	& \mbox{ and }	&	y\left(\Z^d \backslash \neig_g(S)\right) = x\left(\Z^d \backslash \neig_g(S)\right).
\end{eqnarray}
\end{lem}

\begin{proof}
This is a direct application of the definition of strong irreducibility for the configurations $u$ and $x\left(\Z^d \backslash \neig_g(S)\right)$, considering that $\dist\left(S,\Z^d \backslash \neig_g(S)\right) \geq g$.
\end{proof}

\begin{cor}
\label{globAdm}
Given $g \in \N$, there is an algorithm to decide if $u \in \Leng_S(X)$ in time $e^{O(|\partial_g S|\log|\Symb|)}$, for any non-empty $\Z^d$ shift space $X \subseteq \Symb^{\Z^d}$ that satisfies TSSM with gap $g$ and $S \Subset \Z^d$.
\end{cor}

\begin{proof}
By the note after Corollary \ref{TSSMSFT}, we know that $X = \mathsf{X}(\mathcal{F})$ for some $\mathcal{F} \subseteq \Symb^{\romb_g}$.
By Proposition \ref{TSSMperiod}, there exists a periodic point in $X$ of period $2g$ in every direction. Then, by checking all the possible configurations in $\Symb^{[1,2g]^d + \romb_g}$, we can find a periodic point $z$ in time $e^{O(g^d\log|\Symb|)} = O(|\Symb|)$. Given $u \in \Symb^S$, by Lemma \ref{lemPer}, we only need to check that $u$ and $z\left(\partial_{2g+1} \neig_g(S)\right)$ can be extended together to a locally admissible configuration on $\neig_{3g+1}(S)$. It can be checked in time $e^{O(|S|)}$ whether $u$ is locally admissible or not. On the other hand, it can be decided in time $e^{O(|\partial_g S|\log|\Symb|)}$ if there exists a configuration $v \in \Symb^{\partial_g S}$ such that $uvz\left(\partial_{2g+1} \neig_g(S)\right)$ is locally admissible. This is enough for deciding if $u$ is globally admissible or not. Thanks to the discrete isoperimetric inequality $|\partial S| \geq 2d|S|^{\frac{d-1}{d}}$ (this follows directly from the discrete Loomis and Whitney inequality \cite{1-loomis}), we have that $|S| = e^{O(|\partial S|)}$, and we conclude that the total time of the algorithm is $e^{O(|\partial_g S|)\log|\Symb|}$.
\end{proof}

\begin{rem}
It is worthwhile to point that, when $d \geq 3$, there are no known good bounds on the time for checking global admissibility in $\Z^d$ SFTs that only satisfy strong irreducibility.
\end{rem}

\begin{cor}
\label{checkTSSM}
Given $N \in \N$, let $\mathcal{F} \subseteq \Symb^{\romb_N}$ such that $X = \mathsf{X}(\mathcal{F})$ is a non-empty SFT, strongly irreducible with gap $g_0$, for some $g_0 \in \N$. Then, for every $g \geq g_0$, there is an algorithm to check whether $X$ satisfies TSSM with gap $g$ or not, in time $e^{O((g+N)^d\log|\Symb|)}$.
\end{cor}

\begin{proof}
Given the set of configurations $\mathcal{F}  \subseteq \Symb^{\romb_N}$, the algorithm would be the following:
\begin{enumerate}
\item Look for the periodic point provided by Proposition \ref{TSSMperiod}. If such point does not exist, $X$ does not satisfy TSSM with gap $g$. If such point exists, let's denote it by $z$. (This can be done in time $e^{O(g^d\log|\Symb|)}$.)
\item Fix a shape $S \subseteq \romb_{g+N-1} \backslash \{0\}$ and then fix configurations $u \in \Symb^{\{0\}}$, $s \in \Symb^S$ and $v \in \Symb^{\partial_{2N+1}\romb_{g-1} \backslash S}$.
	\begin{enumerate}
	\item Using strong irreducibility with gap $g_0$, check whether $[us]_X$, $[sv]_X$ and $[usv]_X$ are empty or not, by trying to embed $us$, $sv$ and $usv$ in the periodic point $z$ in a locally admissible way (as in Corollary \ref{globAdm}). (This can be done in time $O(|\romb_{2N-1+g+g_0}|)e^{O(|\romb_{2N-1+g+g_0}|\log|\Symb|)} = e^{O(|\romb_{N+g}|\log|\Symb|)}$.)
	\item If $[us]_X = \emptyset$ or $[sv]_X = \emptyset$, continue.
	\item	If $[us]_X,[sv]_X \neq \emptyset$, but $[usv] = \emptyset$, then $X$ does not satisfy TSSM with gap $g$.
	\item If all the cylinders are non-empty, continue.
	\end{enumerate}
\item If after checking all the configurations we have not found $u$, $s$ and $v$ such that  $[us]_X,[sv]_X \neq \emptyset$, but $[usv] = \emptyset$, then $X$ satisfies TSSM with gap $g$ (by Lemma \ref{rombTSSM}).
\end{enumerate}

Then, since $|\romb_{n}| \leq (2n+1)^d$, the total time of this algorithm is $e^{O((g+N)^d\log|\Symb|)}$. 
\end{proof}

The following result is based on a slight modification of the approach used to prove Proposition \ref{algSSF} (see \cite[Proposition 4.1]{1-marcus}), but we include here the whole proof for completeness.

\begin{prop}
\label{algorithmProp}
Let $\Phi$ be a n.n. interaction and $\mu$ an equilibrium state for $\Phi$ such that:
\begin{itemize}
\item $\supp(\mu)$ satisfies TSSM, and
\item $\mu$ satisfies exponential SSM.
\end{itemize}
Then, there is an algorithm to compute $P(\Phi)$ to within $\epsilon$ in time $e^{O\left((\log(1/\epsilon))^{d-1}\right)}$.
\end{prop}

\begin{proof}
Given the values of the n.n. interaction $\Phi$, $\mu$ an equilibrium state for $\Phi$, $X := \supp(\mu)$ an SFT and $\epsilon > 0$, the algorithm would be the following:
\begin{enumerate}

\item Look for a periodic point $z \in X$, provided by Proposition \ref{TSSMperiod}. W.l.o.g., $z$ has period $2g$ in every coordinate direction, for some $g \in \N$. This step does not need the gap of TSSM explicitly, and it does not depend on the value of $\epsilon$.

\item Take $\nu$ the shift-invariant atomic measure supported on the orbit of $z$. From Corollary \ref{TSSMRep}, we have that:
\begin{equation}
P(\Phi) = \int{\left(I_\mu + A_\Phi\right)}d\nu 
= \frac{1}{(2g)^d} \sum_{p \in [1,2g]^d}\left(-\log p_\mu(\sigma_p(z)) + A_\Phi(\sigma_p(z))\right).
\end{equation}

We need to compute the desired approximations of $p_\mu(x)$, for all $x = \sigma_p(z)$ and $p \in [1,2g]^d$. We may assume $p = 0$ (the proof is the same for all $p$). 

\item For $n = 1,2,\dots$, consider the sets $W_n = \romb_n \backslash \past_n$ and $\partial W_n = S_n \sqcup V_n$, where $S_n = \partial W_n \cap \past$ and $V_n = \partial W_n \backslash \past$.

\begin{figure}[ht]
\centering
\includegraphics[scale = 0.65]{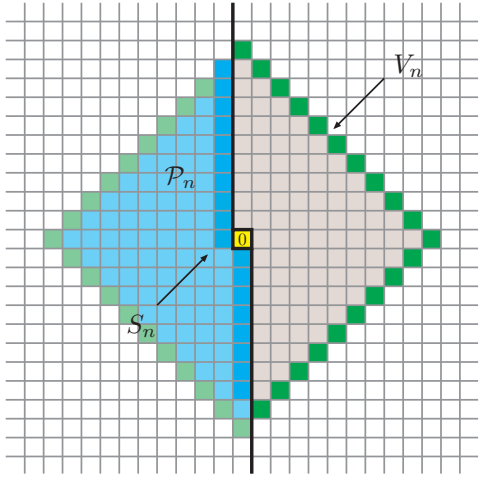}
\caption{Decomposition in the proof of Proposition \ref{algorithmProp}.}
\label{diagram}
\end{figure}

\item Represent $p_\mu(z)$ as a weighted average, using the MRF property:
\begin{equation}
p_\mu(z)	= \sum_{\delta \in \Symb^{V_n}:~\mu(z(S_n)\delta) > 0} \mu\left(z(0) \middle\vert z(S_n)\delta\right)\mu(\delta).
\end{equation}

\item Take $\overline{\delta} \in \argmax_\delta \mu\left(z(0) \middle\vert z(S_n)\delta\right)$ and $\underline{\delta} \in \argmin_\delta \mu\left(z(0) \middle\vert z(S_n)\delta\right)$, over all $\delta \in \Symb^{V_n}$ such that $\mu(z(S_n)\delta) > 0$ (or, since TSSM implies the D-condition, such that $z(S_n)\delta \in \Leng(X)$). Then,
\begin{equation}
\mu\left(z(0) \middle\vert z(S_n)\underline{\delta}\right) \leq p_\mu(z) \leq \mu\left(z(0) \middle\vert z(S_n)\overline{\delta}\right).
\end{equation}

\item By exponential SSM, there are constants $C,\alpha > 0$ such that these upper and lower bounds on $p_\mu(z)$ differ by at most $Ce^{-\alpha n}$. Taking logarithms and considering that $\mu\left(z(0) \middle\vert z(S_n)\underline{\delta}\right) \geq \Cmu > 0$, a direct application of the mean value theorem gives sequences of upper and lower bounds on $\log p_\mu(z)$ with accuracy $e^{-\Omega(n)}$, that is less than $\epsilon$ for sufficiently large $n$.
\end{enumerate}

For $\delta \in \Symb^{V_n}$, the time to compute $\mu\left(z(0) \middle\vert z(S_n)\delta\right)$ is $e^{O(n^{d-1})}$, because this is the ratio of two probabilities of configurations of size $O(n^{d-1})$, each of which can be computed using the transfer matrix method from \cite[Lemma 4.8]{2-marcus} in time $e^{O(n^{d-1})}$. Thanks to Corollary \ref{globAdm}, the necessary time to check if $z(S_n)\delta \in \Leng(X)$ is $e^{O(n^{d-1})}$. Since $\left|\Symb^{V_n}\right| = e^{O(n^{d-1})}$, the total time to compute the upper and lower bounds is $e^{O(n^{d-1})}e^{O(n^{d-1})} = e^{O(n^{d-1})}$.
\end{proof}

\begin{rem}
In the previous algorithm it is not necessary to know explicitly the gap $g$ of TSSM and the constants $C,\alpha > 0$ of the rate $f(n) = Ce^{-\alpha n}$ from exponential SSM.
\end{rem}

\begin{cor}
\label{HighAlg}
Let $\Phi$ be a $\Z^2$ n.n. interaction with $\mu$ an equilibrium state for $\Phi$, such that $\mu$ satisfies SSM with rate $f(n) = Ce^{-\alpha{n}}$, where $\alpha > 4\log|\Symb|$. Then there is an algorithm to compute $P(\Phi)$ to within $\epsilon$ in time $\mathrm{poly}(\frac{1}{\epsilon})$.
\end{cor}

Notice that, in contrast to preceding results, no mixing condition on the support is explicitly needed in Corollary \ref{HighAlg}.

\section*{Acknowledgements}
I would like to thank my advisor, Prof. Brian Marcus, for his guidance and support over all the development of this work. His insights, corrections and suggestions were an invaluable contribution. I would also like to thank Prof. Ronnie Pavlov for his important help in the construction of the family $X_g^d$ and for introducing me to coupling techniques for proving SSM, and Nishant Chandgotia for helpful discussions regarding $k$-checkerboards and the generalized pivot property.

\bibliographystyle{amsplain}
\bibliography{references}

\end{document}